\documentclass[12pt]{amsart}
\renewcommand*{\marginpar}[1]{} 
\usepackage[all]{xy}

\usepackage{amsmath, amsthm, amssymb,latexsym, graphics}
\usepackage[english]{babel}
\usepackage[T1]{fontenc}
\usepackage{amsfonts}
\usepackage{enumerate}
\usepackage{vmargin}
\usepackage{mathrsfs}
\usepackage{mathtools}
\usepackage{lmodern}
\usepackage[pagebackref=true]{hyperref}

\newcommand{\N}{\mathbb{N}}
\newcommand{\Z}{\mathbb{Z}}

\newcommand{\R}{\mathbb{R}}
\newcommand{\C}{\mathbb{C}}
\newcommand{\E}{\mathbb{E}}

\newcommand{\spr}[2]{\langle #1, #2 \rangle}
\newcommand{\ck}{\check{\phantom{i}}}

\newcommand{\HI}{H^\infty}

\newcommand{\Sobolev}{W}

\newcommand{\Sob}{\Sobolev_p}
\newcommand{\Soa}{\Sobolev^\alpha_p}

\newcommand{\Wor}{\widetilde{\mathcal{H}}}
\newcommand{\Wa}{\Wor^\alpha_p}

\newcommand{\Hor}{\mathcal{H}}
\newcommand{\Ha}{\Hor^\alpha_p}
\DeclareMathOperator{\BIP}{BIP}
\newcommand{\Sobexp}{\mathcal{W}}
\newcommand{\Sea}{\Sobexp^\alpha_p}

\newcommand{\equi}{\psi}

\newcommand{\dyad}{\varphi}

\newcommand{\ta}{\langle t \rangle^\alpha}
\newcommand{\tma}{\langle t \rangle^{-\alpha}}

\DeclareMathOperator{\supp}{supp}
\DeclareMathOperator{\Id}{id}
\DeclareMathOperator{\Rad}{Rad}

\DeclareMathOperator{\Str}{Str}
\DeclareMathOperator{\Hol}{Hol}
\DeclareMathOperator{\type}{type}
\DeclareMathOperator{\cotype}{cotype}

\let\Re=\relax \DeclareMathOperator{\Re}{Re}
\let\Im=\relax \DeclareMathOperator{\Im}{Im}

\newcommand{\bignorm}[1]{\Bigl\Vert#1\Bigr\Vert}

\newtheorem{thmalt}{Theorem}[section]

\theoremstyle{definition}
\newtheorem{rem}[thmalt]{Remark}
\newtheorem{defi}[thmalt]{Definition}
\newtheorem{thm}[thmalt]{Theorem}
\newtheorem{cor}[thmalt]{Corollary}
\newtheorem{lem}[thmalt]{Lemma}
\newtheorem{prop}[thmalt]{Proposition}
\newtheorem{exa}[thmalt]{Example}
\newtheorem{assumption}[thmalt]{Assumption}

\numberwithin{equation}{section}

\title[Spectral multiplier theorems via $\HI$ calculus and $R$-bounds]
 {Spectral multiplier theorems via $\HI$ calculus and $R$-bounds} 

\author[Ch. Kriegler]{Christoph Kriegler}
\address{Christoph Kriegler\\
Laboratoire de Math\'ematiques (CNRS UMR 6620)\\
Universit\'e Blaise-Pascal (Clermont-Ferrand 2)\\
Campus Universitaire des C\'ezeaux\\
3, place Vasarely\\
TSA 60026\\
CS 60026\\
63 178 Aubi\`ere Cedex\\
France
}
\email{christoph.kriegler@math.univ-bpclermont.fr}
\thanks{The first named author acknowledges financial support from the Franco-German University (DFH-UFA) and the Karlsruhe House of Young Scientists (KHYS). The second named author acknowledges support from CRC 1173, DFG (Deutsche Forschungsgemeinschaft)}

\author[L. Weis]{Lutz Weis}
\address{Lutz Weis\\
Karlsruher Institut f\"ur Technologie\\
Fakult\"at f\"ur Mathematik\\
Institut f\"ur Analysis\\
Englerstra\ss{}e 2\\
76131 Karlsruhe\\
Germany}
\email{lutz.weis@kit.edu}

\date{\today}

\allowdisplaybreaks

\begin{document}


\begin{abstract}
We prove spectral multiplier theorems for H\"ormander classes $\Hor^\alpha_p$ for $0$-sectorial operators $A$ on Banach spaces assuming a bounded $\HI(\Sigma_\sigma)$ calculus for some $\sigma \in (0,\pi)$ and norm and certain $R$-bounds on one of the following families of operators:
the semigroup $e^{-zA}$ on $\C_+,$ the wave operators $e^{isA}$ for $s \in \R,$ the resolvent $(\lambda - A)^{-1}$ on $\C \backslash \R,$ the imaginary powers $A^{it}$ for $t \in \R$ or the Bochner-Riesz means $(1 - A/u)_+^\alpha$ for $u > 0.$
In contrast to the existing literature we neither assume that $A$ operates on an $L^p$ scale nor that $A$ is self-adjoint on a Hilbert space.
Furthermore, we replace (generalized) Gaussian or Poisson bounds and maximal estimates by the weaker notion of $R$-bounds, which allow for a unified approach to spectral multiplier theorems in a more general setting.
In this setting our results are close to being optimal.
Moreover, we can give a characterization of the ($R$-bounded) $\Hor^\alpha_1$ calculus in terms of $R$-boundedness of Bochner-Riesz means.
\end{abstract}

\maketitle

\section{Introduction}\label{Sec 1 Intro}

Classical spectral multiplier theorems go back to Mihlin and H\"ormander \cite{Ho} who proved that ``Fourier multiplier operators''
\begin{equation}
\label{equ-Fourier-multiplier}
u \mapsto f(-\Delta)u(\cdot) = \mathcal{F}^{-1}[f(|\xi|^2)\hat{u}(\xi)](\cdot)
\end{equation}
are bounded on $L^q(\R^d),\: 1 < q < \infty$ if $f$ satisfies a smoothness condition such as
\begin{equation}\label{equ-intro-hor}
\sup_{R > 0} \int_{R/2}^{2R} \left| t^k f^{(k)}(t) \right|^2 \frac{dt}{t} < \infty \quad (k = 0,1,\ldots,\alpha), \: \alpha = \lceil d/2 \rceil.
\end{equation}
Let us denote by $\Hor^\alpha_2$ the ``H\"ormander class'' of all functions $f : \R_+ \to \C$ satisfying \eqref{equ-intro-hor}.
Then the Fourier multiplier theorem of H\"ormander says that the mapping
\[ f \in \Hor^\alpha_2 \mapsto f(-\Delta) \in B(L^q(\R^d)) \]
defines an algebra homomorphism.
There is by now a large literature extending such spectral multiplier results to Laplace-type operators (including elliptic and Schr\"odinger operators) on $L^q$ spaces on manifolds, Lie groups and graphs \cite{Meda, Alex, MSt, CDMY, Mu, Duon, Chri, DuOS, Bluna, DuSY, COSY,SYY}, see also \cite{Ouha} and the references therein.
Typical consequences of spectral multiplier theorems for $f \in \Hor^\alpha_2$ and an operator $A$ on $L^q$ are norm estimates for important operator families attached to $A,$ e.g. the norm boundedness of the sets
\begin{align}
& \left\{ \left( \frac{\pi}{2} - |\theta| \right)^{\alpha} \exp(-te^{i\theta} A) :\: t > 0,\: \theta \in \left(-\frac{\pi}{2}, \frac{\pi}{2} \right) \right\} \tag*{$(S)_\alpha$} \label{S_alpha} \\
& \left\{ (1 + |s| A)^{-\alpha} \exp(isA) :\: s \in \R \right\} & \text{ Wave operators}\tag*{$(W)_\alpha$} \label{W_alpha} \\
& \left\{ (1 + |t|)^{-(\alpha + \epsilon)} A^{it} :\: t \in \R \right\} & \text{ Bounded imaginary powers} \tag*{$(\BIP)_\alpha$} \label{BIP_alpha} \\
& \left\{ (1 - A/u)^{\alpha + \epsilon - \frac12}_+ :\: u > 0 \right\} & \text{ Bochner-Riesz means} \tag*{$(\text{BR})_\alpha$} \label{BR_alpha}
\end{align}
for any $\epsilon > 0.$
Conversely, a common strategy to establish spectral multiplier theorems is to start from one of these families and estimate the corresponding representation formulas
\begin{align}
f(A) & = \frac{1}{2\pi} \int_\R \check{f}(t) e^{-itA} dt, & \check{f} \text{ the inverse Fourier-transform} \nonumber \\
f(A) & = \frac{1}{2\pi} \int_\R \mathfrak{M}(f)(s) A^{is} ds , & \mathfrak{M}(f) \text{ the Mellin-transform} \label{equ-inversion-formulas} \\
f(A) & = \frac{(-1)^{\lfloor \alpha \rfloor}}{\Gamma(\alpha)} \int_0^\infty f^{(\alpha)}(u)(u-A)^{\alpha - 1}_+ udu, & f^{(\alpha)}\text{ the (fractional) derivative} \nonumber
\end{align}
to obtain bounded operators $f(A)$ on $L^q(U)$ for all $f \in \Hor^\alpha_2.$

Usually, in a first step, one estimates one of the equations in \eqref{equ-inversion-formulas} for a function $f \in C^\infty_c(0,\infty)$ by using (generalized) Gaussian or Poisson kernel bounds or maximal estimates for the corresponding operator family in \eqref{S_alpha}, \eqref{W_alpha}, \eqref{BIP_alpha} and \eqref{BR_alpha}.
Then for a general $f \in \Hor^\alpha_2,$ one considers a dyadic decomposition $f(\cdot) = \sum_{n \in \Z} f \cdot \phi(2^{-n}\cdot)$ for an appropriate $\phi \in C^\infty(\frac12,2)$ and ``puts the pieces $f \phi(2^{-n} \cdot)$ together again'' by techniques related to the Littlewood-Paley theory.

In this paper, we explore an operator theoretic approach to spectral multiplier theorems.
In particular, we show that the ``Paley-Littlewood arguments'' for singular integrals can be replaced by a localization argument using the holomorphic $\HI$-calculus of the operator $A$ (see Section \ref{Sec Hor Mih classes}).
For many Laplace type operators, the boundedness of the $\HI(\Sigma_\sigma)$ calculus is already well known.
The various kernel bounds such as (generalized) Gaussian and Poisson bounds or maximal estimates which are commonly used in the first step, we replace by $R$-bounds:
A set $\tau$ of operators on an $L^q(U)$-space is called $R$-bounded if there is a constant $C$ such that for all $T_1,\ldots,T_n \in L^q(U)$ and $x_1,\ldots,x_n \in L^q(U),$ we have
\[ \left\| \left( \sum_{j = 1}^n |T_j x_j|^2 \right)^{\frac12} \right\|_{L^q(U)} \leq C \left\| \left( \sum_{j = 1}^n |x_j|^2 \right)^{\frac12} \right\|_{L^q(U)} \]
(the smallest $C$ in this inequality will be the $R$-bound $R(\tau)$ of $\tau$).
Clearly such estimates are intimately related to Littlewood-Paley theory and it is known that many (generalized) Gaussian- and Poisson estimates, also on metric measure spaces with the doubling property (see e.g. \cite{Blun,BlKua,KuWe}), as well as maximal estimates imply the above $R$-boundedness condition.

Therefore our operator theoretic approach allows us to present a unified approach to spectral multiplier theorems.
Furthermore, we are not restricted to the $L^q$-scale and self-adjoint operators on $L^2,$ but we can formulate theorems for $0$-sectorial operators on a Banach space.
Here are some samples.
(From now on we assume $\alpha \in (0,\infty),$ not just $\alpha \in \N.$)

\begin{thm}
\label{thm-intro-sufficient}
(see Theorems \ref{Thm Sufficient conditions BIP},\ref{Thm Sufficient conditions Sgr}, \ref{Thm Sufficient conditions Wave}).
Let $A$ be a $0$-sectorial operator on a space $L^q(U),\: 1 < q < \infty$ (more generally on a Banach space $X$ with Pisier's property $(\alpha)$).
Suppose furthermore that $A$ has a bounded $\HI(\Sigma_\omega)$ calculus for some $\omega \in (0,\frac{\pi}{2}).$
\begin{enumerate}
\item The $R$-boundedness of one of the sets \ref{S_alpha},\ref{W_alpha} or \ref{BIP_alpha} above implies that $A$ has an $R$-bounded $\Hor^\beta_2$ calculus for $\beta > \alpha + \frac12$ on $L^q(U),$ i.e. the set
\begin{equation}
\label{equ-intro-R-bdd-calculus}
 \left\{ f(A) :\: \|f\|_{\Hor^\beta_2} \leq 1 \right\} \text{ is }R\text{-bounded}.
\end{equation}
\item Conversely, \eqref{equ-intro-R-bdd-calculus} implies that each of the sets \ref{S_alpha}, \ref{W_alpha}, \ref{BIP_alpha} is $R$-bounded with $\alpha \geq \beta$ ($\alpha > \beta$ for the imaginary powers).
\end{enumerate}
\end{thm}

Here we used for $\alpha > \frac1p$ the notation
\[ \Hor^\alpha_p = \{ f \in L^p_{loc}(\R_+) :\: \|f\|_{\Hor^\alpha_p} = \sup_{t > 0} \| \phi f(t \cdot) \|_{W^\alpha_p(\R)} < \infty \} , \]
where $\phi$ is a non-zero $C^\infty_c(0,\infty)$ function (different choices resulting in equivalent norms) and $W^\alpha_p(\R)$ stands for the usual Sobolev space.

Some variants of this theorem concerning $\Hor^\beta_p$ calculi with $p \not= 2$ and ``dyadic'' bounds will be discussed in Sections \ref{Sec 6 Hormander} and \ref{sec-Semigroup}.
Essentially, these theorems say that in the presence of a bounded $\HI$-calculus, $\Hor^\beta_2$ spectral multiplier theorems follow if one strengthens the norm bounds in \ref{S_alpha}, \ref{W_alpha}, \ref{BIP_alpha} to $R$-bounds.
We also show that the norm bounds for \ref{S_alpha} or \ref{W_alpha} by themselves are not strong enough to ensure spectral multiplier theorems (see Subsection \ref{subsec-Kalton}).
For \ref{BIP_alpha} we have a positive result in Theorem \ref{Thm Sufficient conditions BIP}.
For functional calculi derived from the norm bounds in \ref{S_alpha}, see \cite{GaPy}.
Due to the generality of our approach we do not always obtain the best possible exponent $\beta$ for a given operator $A$ with additional structure.
However in our general setting our assumptions, in particular the gap between the parameter $\alpha$ in \ref{S_alpha}, \ref{W_alpha} or \ref{BIP_alpha} to the order $\beta$ of the H\"ormander calculus are close to being optimal; we discuss this in Section \ref{Sec Examples Counterexamples}.
Moreover, in the case of Bochner-Riesz means we obtain a characterization of the $\Hor^\alpha_1$ calculus in terms of $R$-bounds as follows.

\begin{thm}
\label{thm-intro-BR}
Let $A$ be as in Theorem \ref{thm-intro-sufficient} above.
\begin{enumerate}
\item If $A$ has an $R$-bounded $\Hor^\beta_1$ calculus, then 
\begin{equation}
\label{equ-intro-BR}
\{ (1 - A/u)^{\alpha-1}_+ :\: u > 0 \}\text{ is }R\text{-bounded}
\end{equation}
for all $1 < \beta < \alpha.$
\item Conversely, \eqref{equ-intro-BR} implies that $A$ has an $R$-bounded $\Hor^\alpha_1$ functional calculus.
\end{enumerate}
\end{thm}

Our framework also allows us to work out a rather general version of the Paley-Littlewood theory for operators with a $\Hor^\alpha_p$ calculus, including estimates of the form (for $X = L^q(U)$):
\begin{align}
& \left\| \left( \sum_{n \in \Z} \left| \phi(2^n A) x \right|^2 \right)^{\frac12} \right\| \cong \|x\|_{L^q}, \label{equ-square-function-1} \\
& \left\| \left( \int_0^\infty |\phi(tA) x|^2 \frac{dt}{t} \right)^{\frac12} \right\| \cong \|x\|_{L^q}, \label{equ-square-function-2}
\end{align}
and their analogues in Banach spaces, see \cite{KrW1}.
For precise characterizations of the exponents of the H\"ormander calculus in terms of $R$-bounds and square functions of the form \eqref{equ-square-function-1}, \eqref{equ-square-function-2}, also in Banach spaces, see \cite{KrW2,Kr2}.

We end this introduction with an overview of the article.
In Section \ref{Sec Prelims Rad} we recall the definitions of $R$-boundedness and its variants and give some abstract results used in the main part of the paper.
Section \ref{Sec Hor Mih classes} contains the background on the holomorphic functional calculus as well as several function spaces related to \eqref{equ-intro-hor}.
That is, we introduce the Sobolev and H\"ormander function spaces $\Sea$ and $\Ha$ as well as the associated functional calculi, and show some simple properties used in this paper.
In Section \ref{sec-extended-Hoermander-calculus}, we introduce and study an auxiliary H\"ormander calculus.
It allows to define $f(A)$ for $f$ locally belonging to $\Ha$ in a precised sense, under several possible weak assumptions.
This is useful if there is no a priori self-adjoint calculus at hand, i.e. the underlying Banach space is not an $L^p$ space, on which $f(A)$ would be extended by density of $L^2 \cap L^p.$
The issue of obtaining estimates of $f(A)$ for $f \in \Hor^\alpha_p$ from estimates for $C^\infty_c(0,\infty)$ functions is addressed by means of a localization procedure of the support of $f.$
Namely in Section \ref{Sec 5 Mihlin}, we show that under the presence of an $\HI$ calculus, an $R$-bounded $\Sea$ calculus extends automatically to a $\Ha$ calculus (see Theorem \ref{Thm Localization Principle}). 
In Sections \ref{Sec 6 Hormander} and \ref{sec-Semigroup}, we discuss the full $\Ha$ calculus and prove in particular the above Theorem \ref{thm-intro-sufficient}.
In Section \ref{Sec Examples Counterexamples}, we discuss to what extent
the assumptions of Theorem \ref{thm-intro-sufficient} are optimal
and give several examples and counterexamples in connection with Theorem \ref{thm-intro-sufficient}.
In Section \ref{Sec Bochner-Riesz}, we prove Theorem \ref{thm-intro-BR}.
Finally, in Section \ref{Sec Strip-type Operators}, we show how the results from Sections \ref{Sec 6 Hormander}, \ref{sec-Semigroup} and \ref{Sec Bochner-Riesz} can be transferred to bisectorial operators.
We also look at strip-type operators, which generate polynomially bounded groups and give a sketch of their theory of H\"ormander type functional calculus using the results from the preceding sections.

\section{$R$-bounded sets of operators}\label{Sec Prelims Rad}

A classical theorem of Marcinkiewicz and Zygmund states that for elements $x_1,\ldots,x_n \in L^p(U,\mu)$ we can express ``square sums'' in terms of random sums
\[ \left\| \left( \sum_{j=1}^n |x_j(\cdot)|^2 \right)^{\frac12} \right\|_{L^p(U)}
\cong \left( \E \| \sum_{j=1}^n \epsilon_j x_j \|_{L^p(U)}^q \right)^{\frac1q}
\cong \left( \E \| \sum_{j=1}^n \gamma_j x_j \|_{L^p(U)}^q \right)^{\frac1q} \]
with constants only depending on $p,q \in [1,\infty).$
Here $(\epsilon_j)_j$ is a sequence of independent Bernoulli random variables (with $P(\epsilon_j = 1) = P(\epsilon_j = -1) =\frac12$) and $(\gamma_j)_j$ is a sequence of independent standard Gaussian random variables.
Following \cite{Bou} it has become standard by now to replace square functions in the theory of Banach space valued function spaces by such random sums (see e.g. \cite{KuWe}).
Note however that Bernoulli sums and Gaussian sums for $x_1,\ldots,x_n$ in a Banach space $X$ are only equivalent if $X$ has finite cotype (see \cite[p.~218]{DiJT} for details).

\begin{defi}
Let $\tau$ be a subset of $B(X,Y),$ where $X$ and $Y$ are Banach spaces.
We say that $\tau$ is $R$-bounded if there exists a $C < \infty$ such that
\[ \E \bignorm{ \sum_{k=1}^n \epsilon_k T_k x_k } \leq C \E \bignorm{ \sum_{k=1}^n \epsilon_k x_k } \]
for any $n \in \N,$ $T_1,\ldots, T_n \in \tau$ and $x_1,\ldots,x_n \in X.$
The smallest admissible constant $C$ is denoted by $R(\tau).$
\end{defi}

Recall that by definition, $X$ has Pisier's property $(\alpha)$ if for any finite family $x_{k,l}$ in $X,$ $(k,l) \in F,$ where $F \subset \Z \times \Z$ is a finite array,
we have a uniform equivalence
\[ \E_\omega \E_{\omega'} \bignorm{ \sum_{(k,l) \in F} \epsilon_k(\omega) \epsilon_l(\omega') x_{k,l} }_{X} \cong \E_\omega \bignorm{ \sum_{(k,l) \in F} \epsilon_{k,l}(\omega) x_{k,l} }_{X}. \]
Note that property $(\alpha)$ is inherited by closed subspaces, and that an $L^p$ space has property $(\alpha)$ provided $1 \leq p < \infty$ \cite[Section 4]{KuWe}.

Recall that by definition, $X$ has type $p \in [1,2]$ (resp. cotype $q \in [2,\infty]$) if there is a uniform estimate for any finite family $x_1,\ldots,x_n$ in $X$
\[ \E \bignorm{ \sum_{k} \epsilon_k x_k }_{X} \lesssim \left( \sum_k \|x_k\|^p \right)^{\frac1p}
 \text{ resp. } \left( \sum_k \|x_k\|^q \right)^{\frac1q} \lesssim \E\bignorm{ \sum_k \epsilon_k x_k }_{X} \]
(standard modification if $q = \infty$).
An $L^p$ space for $1 \leq p < \infty$ has always type $\min(2,p)$ and cotype $\max(2,p)$ \cite[p.~219]{DiJT}.

\begin{defi}
Let $\tau \subset B(X,Y).$
Then $\tau$ is called semi-$R$-bounded if there exists a $C < \infty$ such that
\[ \left( \E \left\| \sum_{k=1}^n \epsilon_k T_k a_k x \right\|^2 \right)^{\frac12} \leq C \left( \sum_{k=1}^n |a_k|^2 \right)^{\frac12} \|x\| \]
for any $n \in \N,\:T_1,\ldots,T_n \in \tau,\:a_1,\ldots,a_n \in \C$ and $x \in X.$
The least admissible constant $C$ is denoted by $R_s(\tau).$
One clearly has that any $R$-bounded set is semi-$R$-bounded and $R_s(\tau) \leq R(\tau).$
\end{defi}

We have the following characterization of semi-$R$-boundedness in terms of $R$-boundedness.

\begin{lem}\cite[Lemma 2.1]{VeWe}\label{Lem semi-R-bounded}
A set $\tau \subset B(X,Y)$ is semi-$R$-bounded with $R_s(\tau) \leq M$ if and only if for all $x \in X,$ the set $\tau_x = \{ Tx \in B(\C,X) :\: T \in \tau \}$ is $R$-bounded with $R(\tau_x) \leq M \|x\|.$
\end{lem}

\begin{rem}
\label{rem-semi-R-bounded}
For a family $\tau \subset B(X),$ one always has $R(\tau) \geq \sup_{T \in \tau} \|T\|$ and equivalence holds if and only if $X$ is isomorphic to Hilbert space.
On the other hand, if $X$ or $X'$ has type $2,$ then any bounded family $\tau$ is automatically semi-$R$-bounded,
which motivates the above definition.
Indeed, this follows from \cite[Proposition 2.2]{VeWe} and a dualization argument for the case that $X'$ has type $2.$
Note that if $\tau \subset B(X,Y)$ is semi-$R$-bounded, then $\{T' :\: T \in \tau \}$ is semi-$R$-bounded in $B(Y',X').$
\end{rem}

An important result is the following proposition due to Hyt\"onen and Veraar which produces $R$-bounded sets by integrating against $L^{r'}(\R)$ functions.

\begin{prop}\cite[Proposition 4.1, Remark 4.2]{HyVe}\label{Prop Hytonen Veraar}
\begin{enumerate}
\item Let $X,Y$ be Banach spaces and $(\Omega,\mu)$ a $\sigma$-finite measure space.
Let $r \in [1,\infty)$ satisfy $\displaystyle \frac{1}{r} > \frac{1}{\type Y} - \frac{1}{\cotype X}.$
Further let $T \in L^r(\Omega,B(X,Y))$ or assume merely that $\omega \in \Omega \mapsto N(t) \in B(X,Y)$ is strongly measurable with $\|T(\cdot)\|_{B(X,Y)}$ dominated by a function in $L^r(\Omega).$
Denote $r'$ the conjugated exponent to $r.$

Then the set
\[ \tau = \{ T_f :\: \|f\|_{L^{r'}(\Omega)} \leq 1 \} \]
is $R$-bounded, where $T_f x = \int_\Omega f(\omega) T(\omega)x d\omega$ and $R(\tau) \lesssim \|T\|_{L^r(\Omega,B(X,Y))}.$
\item Let $Y$ have type $p$ and let $r \in [1,\infty)$ satisfy $\frac1r > \frac1p - \frac12.$
If $\omega \in \Omega \mapsto N(t) \in B(X,Y)$ is strongly measurable and there exists $C < \infty$ such that
$\left(\int_\Omega \|N(t) x\|^r dt\right)^{\frac1r} \leq C \|x\|,$ then the set
\[ \tau = \{ T_f :\: \|f\|_{L^{r'}(\Omega)} \leq 1 \} \]
is semi-$R$-bounded and $R_s(\tau) \lesssim C.$
\end{enumerate}
\end{prop}

\begin{proof}
Part (1) is proved in \cite[Proposition 4.1, Remark 4.2]{HyVe}.
Then part (2) follows from part (1) and Lemma \ref{Lem semi-R-bounded} in the following way.
$\C$ has cotype $2,$ $Y$ has type $p,$ and by assumption, $\frac1r > \frac1p - \frac12.$
Then part (1) yields that $R(\{T_f x :\: \|f\|_{L^{r'}(\Omega)} \leq 1 \}) \leq M \|x\|$ in $B(\C,Y).$
Thus by Lemma \ref{Lem semi-R-bounded}, $R_s(\{T_f:\:\|f\|_{L^{r'}(\Omega)} \leq 1 \}) \leq M$ in $B(X,Y).$
\end{proof}

To obtain $R$-bounds, it sometimes suffices to have simple norm bounds for an analytic operator family, which autoimproves to an $R$-bounded version.
This is precised in the following lemma.

\begin{lem}\label{Lem norm to R bound analytic}
Assume that $X$ has type $p$ and cotype $q.$
Let $F : \C_+ \to B(X)$ be an analytic function such that $\|F(z)\| \leq C \left( \frac{|z|}{\Re z} \right)^\alpha$ for any $\Re z > 0$ and some $\alpha \geq 0.$
Then for $\delta > \frac1p - \frac1q$ there is a constant $C < \infty$ such that we have
\[ \left\{ \left( \frac{\Re z}{|z|} \right)^{\alpha + \delta} F(z) :\: \Re z = \epsilon \right\} \]
is $R$-bounded for any $\epsilon > 0,$ with $R$-bound less than $C.$
\end{lem}

\begin{proof}
For $\lambda = 2 \epsilon + i s$ we have by the Cauchy integral formula
\[ \frac{F(\lambda)}{\lambda^{\alpha + \delta}} = \frac{1}{2\pi i} \int_{\Re z = \epsilon} \frac{1}{z - \lambda} \frac{F(z)}{z^{\alpha + \delta}} dz. \]
Choose $\delta > \frac1r > \frac1p - \frac1q.$
Then
\[ \left( \int_{\Re z = \epsilon} \left| \frac{1}{z - \lambda} \right|^{r'} dz \right)^{\frac{1}{r'}} = \left( \int_\R \frac{1}{(|t-s|^2 + \epsilon^2)^{r'/2}} dt \right)^{\frac{1}{r'}} = \left( \int_\R \frac{1}{(1 + |t/\epsilon|^2)^{r'/2}} \frac{dt}{\epsilon} \right)^{\frac{1}{r'}} \frac{\epsilon^{\frac{1}{r'}}}{\epsilon}. \]
Furthermore,
\begin{align*} \left( \int_{\Re z = \epsilon} \left\| \frac{F(z)}{z^{\alpha+\delta}} \right\|^r dz \right)^{\frac1r} & \leq \sup_{\Re z = \epsilon} \left\| \frac{F(z)}{z^\alpha} \right\| \left( \int_{\Re z = \epsilon} \frac{1}{|z|^{r \delta}} dz \right)^{\frac1r} \leq C \epsilon^{-\alpha} \left( \int_\R \frac{1}{(\epsilon^2 + |t|^2)^{r \delta/2}} dt \right)^{\frac1r} \\
& = C \epsilon^{-\alpha} \epsilon^{-\delta} \epsilon^{\frac1r} \left( \int_\R \frac{1}{(1+|t|^2)^{r \delta/2}} dt \right)^{\frac1r}.
\end{align*}
Hence,
\[ \left( \int_{\Re z = \epsilon} \left| \frac{1}{z - \lambda} \right|^{r'} dz \right)^{\frac{1}{r'}} \left( \int_{\Re z = \epsilon} \left\| \frac{F(z)}{z^{\alpha + \delta}} \right\|^r dz \right)^{\frac1r} \leq C \frac{\epsilon^{\frac{1}{r'}}}{\epsilon} \epsilon^{-\alpha -\delta} \epsilon^{\frac1r} \cong C ( \Re \lambda )^{-\alpha-\delta}. \]
By Proposition \ref{Prop Hytonen Veraar} and the fact that $\frac1r > \frac1p - \frac1q,$ it follows that $\{ F(z) \left( \frac{\Re z}{|z|} \right)^{\alpha + \delta} :\: \Re z = \epsilon \}$ is $R$-bounded, with a uniform $R$-bound in $\epsilon > 0.$
\end{proof}

As a corollary, we record

\begin{cor}\label{Cor norm to R bound analytic}
Let $X$ be a Banach space with type $p,$ cotype $q$ and let $\frac1r > \frac1p - \frac1q.$
Let $-A$ generate an analytic semigroup $(e^{-zA})_{\Re z > 0}.$
\begin{enumerate}
\item If $\| e^{-zA} \| \leq C \left( \frac{|z|}{\Re z} \right)^\alpha$ for $\Re z > 0,$ then there is a constant $C < \infty$ such that
\[ R\left(\left\{ \left( \frac{\Re z}{|z|} \right)^\beta e^{-zA} :\: \Re z = \epsilon \right\} \right)  \leq C \text{ for }\beta > \alpha + \frac1r\text{ and any }\epsilon > 0. \]
\item If $\|z (z - A)^{-1}\| \leq C \frac{|z|^\alpha}{| \Im z |^\alpha}$ for $\Im z > 0,$ then there is a $C < \infty$ such that
\[R\left(\left\{\left(\frac{|\Im z|}{|z|}\right)^{\beta} z(z-A)^{-1} :\: \Im z = \epsilon \right\} \right) \leq C \text{ for }\beta > \alpha + \frac1r\text{ and any }\epsilon \neq 0.\]
\end{enumerate}
\end{cor}

\begin{proof}
For (1), we set $F(z) = \exp(-zA)$ and apply Lemma \ref{Lem norm to R bound analytic}, whereas for (2), we set both $F(z) = iz (iz - A)^{-1}$ and $F(z) = - iz (-iz - A)^{-1}$ and apply Lemma \ref{Lem norm to R bound analytic} twice.
\end{proof}

In the claim of Lemma \ref{Lem norm to R bound analytic}, one cannot replace the vertical axes $\Re z = \epsilon$ by the right half plane $\Re z > 0.$
This follows from the following counterexample.

\begin{exa}
Let $A$ be the negative generator of a bounded analytic semigroup which is not $R$-sectorial (see the beginning of Section \ref{Sec Hor Mih classes} for the definition of $R$-sectoriality), i.e. for some $\delta \in (0,\frac{\pi}{2}),$ $\{\exp(-zA) :\: z \in \Sigma_\delta \}$ is bounded, but for no $\delta \in (0,\frac{\pi}{2})$ is $\{\exp(-zA):\: z \in \Sigma_\delta\}$ $R$-bounded.
For an example of such a semigroup on an $L^p$-space, see \cite[Theorem 6.5]{Fackler} and \cite[2.20 Theorem]{KuWe}.
Put $a = \delta/\frac{\pi}{2} \in (0,1)$ and set $F(z) = \exp(-z^{a}A).$
Then the assumptions of Lemma \ref{Lem norm to R bound analytic} hold with $\alpha = 0.$
If $\left\{ \left( \frac{\Re z}{|z|} \right)^{\beta} F(z) :\: \Re z > 0 \right\}$ were $R$-bounded for some $\beta > 0,$ then $\left\{ F(z):\: z \in \Sigma_{\frac{\pi}{4}} \right\}$ would be $R$-bounded, so $\{\exp(-zA):\: z \in \Sigma_{\delta/2} \}$ would be $R$-bounded.
This is a contradiction.
\end{exa}

The following result of van Gaans will be useful.

\begin{prop}\cite[Theorem 3.1]{vG}\label{Prop van Gaans}
Let $X,Y$ be Banach spaces such that $X$ has cotype $q$ and $Y$ has type $p.$
Let $\tau_1,\tau_2,\ldots$ be $R$-bounded sets in $B(X,Y)$ such that $C = \left(\sum_{k=1}^\infty R(\tau_k)^r \right)^{\frac1r}$ is finite with $\frac1r = \frac1p - \frac1q.$
Then the union $\tau = \bigcup_{k=1}^\infty \tau_k$ is $R$-bounded with $R(\tau) \lesssim C.$
\end{prop}

As a corollary, we have the following further method to pass from bounded sets to (semi-)$R$-bounded ones.

\begin{cor}
Let $\R \ni t \mapsto U(t) \in B(X)$ a (not necessarily strongly continuous) one parameter group on a Banach space $X$ with type $p$ and cotype $q$ and let $X'$ have type $p'.$
Assume that $\{(1+|t|)^{-\alpha} U(t) :\: t \in \R\}$ is bounded for some $\alpha \geq 0.$
\begin{enumerate}
\item Assume that $\{U(t):\: t\in [0,1]\}$ is $R$-bounded.
Then $\{(1+|t|)^{-\beta} U(t) :\: t \in \R \}$ is $R$-bounded for $\beta > \alpha + \frac1p - \frac1q.$
\item Assume that $\{U(t):\: t\in [0,1]\}$ is semi-$R$-bounded.
Then $\{(1+|t|)^{-\beta} U(t):\: t \in \R \}$ is semi-$R$-bounded for $\beta > \alpha + \min(\frac1p,\frac{1}{p'}) - \frac12.$
\end{enumerate}
\end{cor}

\begin{proof}
For part (1), we write 
\[\{ ( 1+ |t|)^{-\beta} U(t):\: t \in \R \} \subseteq C \text{conv} \left(\bigcup_{n \in \Z} \{U(t) :\: t \in [0,1] \} \circ \{ (1+|n|)^{-(\beta - \alpha)} (1 + |n|)^{-\alpha} U(n) \}\right),\]
where conv stands for the convex hull.
By \cite[2.13 Theorem]{KuWe}, taking the convex hull does not increase the $R$-bound.
Note that the last set is $R$-bounded as a singleton with $R$-bound $\lesssim (1 + |n|)^{-(\beta-\alpha)},$ which is $\ell^r(\Z)$-summable for $\frac1r = \frac1p - \frac1q.$
Since $\{ U(t):\: t \in [0,1] \}$ is $R$-bounded by assumption, Proposition \ref{Prop van Gaans} yields the claim.

For part (2), we argue similarly; note that the composition of a semi-$R$-bounded set after a singleton is again semi-$R$-bounded.
We use Proposition \ref{Prop van Gaans} together with Lemma \ref{Lem semi-R-bounded}, a dualization argument if $p' > p$ and the fact that $U(t)'$ is again a one parameter group.
\end{proof}

\section{The $\HI$ and H\"ormander calculus}\label{Sec Hor Mih classes}

Our approach to the H\"ormander calculus is based on the $\HI$ calculus.

\subsection{$0$-sectorial operators}\label{Subsec A B}

We briefly recall standard notions of sectorial operators and the $\HI$ calculus.
For $\omega \in (0,\pi)$ we let $\Sigma_\omega = \{ z \in \C \backslash \{ 0 \} :\: | \arg z | < \omega \}$ be the sector around the positive half-axis of aperture angle $2 \omega.$
We further define $\HI(\Sigma_\omega)$ to be the space of bounded holomorphic functions on $\Sigma_\omega.$
This space is a Banach algebra when equipped with the norm $\|f\|_{\infty,\omega} = \sup_{\lambda \in \Sigma_\omega} |f(\lambda)|.$

A closed operator $A : D(A) \subset X \to X$ is called $\omega$-sectorial, if the spectrum $\sigma(A)$ is contained in $\overline{\Sigma_\omega},$ $R(A)$ is dense in $X$ and
\begin{equation}\label{Equ Def Sectorial}
\text{for all }\theta > \omega\text{ there is a }C_\theta > 0\text{ such that }\|\lambda (\lambda - A)^{-1}\| \leq C_\theta \text{ for all }\lambda \in \overline{\Sigma_\theta}^c.
\end{equation}
Note that $\overline{R(A)} = X$ along with \eqref{Equ Def Sectorial} implies that $A$ is injective.
In the literature, in the definition of sectoriality, the condition $\overline{R(A)} = X$ is sometimes omitted.
Note that if $A$ satisfies the conditions defining $\omega$-sectoriality except $\overline{R(A)} = X$ on $X = L^p(\Omega),\, 1 < p < \infty$ (or any reflexive space),
then there is a canonical decomposition $X  = \overline{R(A)} \oplus N(A),\,x = x_1 \oplus x_2,$ and $A = A_1 \oplus 0,\,x \mapsto A x_1 \oplus 0,$
such that $A_1$ is $\omega$-sectorial on the space $\overline{R(A)}$ with domain $D(A_1) = \overline{R(A)} \cap D(A).$

For an $\omega$-sectorial operator $A$ and a function $f \in \HI(\Sigma_\theta)$ for some $\theta \in (\omega,\pi)$ that satisfies moreover an estimate $|f(\lambda)| \leq C |\lambda|^\epsilon / | 1 + \lambda |^{2\epsilon},$
one defines the operator
\begin{equation}\label{Equ Cauchy Integral Formula}
f(A) = \frac{1}{2 \pi i} \int_{\Gamma} f(\lambda) (\lambda - A)^{-1} d\lambda ,
\end{equation}
where $\Gamma$ is the boundary of a sector $\Sigma_\sigma$ with $\sigma \in (\omega,\theta),$ oriented counterclockwise.
By the estimate of $f,$ the integral converges in norm and defines a bounded operator.
If moreover there is an estimate $\|f(A)\| \leq C \|f\|_{\infty,\theta}$ with $C$ uniform over all such functions, then $A$ is said to have a bounded $\HI(\Sigma_\theta)$ calculus.
In this case, there exists a bounded homomorphism $\HI(\Sigma_\theta) \to B(X),\,f \mapsto f(A)$ extending the Cauchy integral formula \eqref{Equ Cauchy Integral Formula}.

We refer to \cite{CDMY} for details.
We call $A$ $0$-sectorial if $A$ is $\omega$-sectorial for all $\omega > 0.$
Further, $A$ is called $R$-sectorial if $\{ \lambda (\lambda - A)^{-1} :\: \lambda \in \overline{\Sigma_\theta}^c \}$ is $R$-bounded for some $\theta \in (0,\pi)$ \cite[p.~76]{KuWe}.
In this case, $\omega_R(A)$ is defined to be the infimum over all such $\theta.$
Note that if $X$ has property $(\alpha)$ (see Section \ref{Sec Prelims Rad} for the definition), then a sectorial operator with bounded $\HI$ calculus is always $R$-sectorial \cite[12.8 Theorem]{KuWe}.
For the definition of $R$-boundedness see Section \ref{Sec Prelims Rad}.

To build stronger functional calculi we recall the following function spaces.

\begin{defi}\label{Def Mih Hor}~
\begin{enumerate}
\item
Let $p \in [1,\infty)$ and $\alpha > \frac1p.$
We define 
\[\Sea = \{ f : (0,\infty) \to \C :\: \|f\|_{\Sea} = \|f_e\|_{\Soa} < \infty \}\]
and equip it with the norm $\|f\|_{\Sea}.$
Here we write from now on
\[f_e : J \to \C,\,z \mapsto f(e^z)\]
for a function $f : I \to \C$ such that $I \subset \C \backslash (-\infty,0]$ and $J = \{ z \in \C : \: | \Im z | < \pi,\:e^z \in I \}.$
Moreover,
$\Soa = \{ f \in L^p(\R):\: \|f\|_{\Soa} = \| (\hat{f}(t)(1 + |t|)^\alpha)\ck\|_p < \infty \}.$
The spaces $\Soa$ and $\Sea$ are Banach algebras with respect to pointwise multiplication if $\alpha > \frac1p.$
\item
Let $\psi$ be a fixed function in $C^\infty_c(\R_+)\backslash \{ 0 \}.$
We define the H\"ormander class
\[\Ha = \{ f \in L^p_{\text{loc}}(\R_+) :\: \|f\|_{\Ha} = \sup_{t > 0}\|\psi f(t\cdot)\|_{\Soa} < \infty \}.\]
This definition does not depend on the particular choice of $\psi,$ two different choices giving equivalent norms, see e.g. \cite[p.~445]{DuOS}.
\end{enumerate}
\end{defi}

We have the following elementary properties of H\"ormander spaces.
Their proofs may be found in \cite[Propositions 4.8 and 4.9, Remark 4.16]{Kr}.

\begin{lem}\label{Lem Elementary Mih Hor}
Let $p,q \in [1,\infty).$
\begin{enumerate}
\item The spaces $\Sea$ and $\Ha$ are Banach algebras.
\item
Let $\alpha > \frac1q > \frac1p,$ $\alpha \geq \beta + \frac1q - \frac1p $ and $\sigma \in (0,\pi).$
Then
\[\HI(\Sigma_\sigma) \hookrightarrow \Ha \hookrightarrow \Hor^\alpha_q \hookrightarrow \Hor^\beta_p.\]
In particular, the choice of $p$ in $\Hor^\alpha_p$ is only relevant when one is looking for the best exponent $\alpha.$
\item For any $t > 0,$ we have $\|f\|_{\Ha} = \|f(t \cdot)\|_{\Ha}.$
\end{enumerate}
\end{lem}

\begin{rem}\label{Rem Classical Mih Hor}
The name ``H\"ormander class'' is justified by the following fact.
The classical H\"{o}rmander condition with a parameter $\alpha_1 \in \N$ reads as follows \cite[(7.9.8)]{Hoa}:
\begin{equation}\label{Equ Classical Hoermander condition}
\sum_{k = 0}^{\alpha_1} \sup_{R > 0} \int_{R/2}^{2R} |R^k f^{(k)}(t)|^p dt/R < \infty.
\end{equation}
\end{rem}

By the following lemma which is proved in \cite[Proposition 4.11]{Kr}, the norm $\|\cdot\|_{\Ha}$ expresses condition \eqref{Equ Modern Hoermander condition}
and generalizes the classical H\"{o}rmander condition \eqref{Equ Classical Hoermander condition}.

\begin{lem}\label{Lem Classical and modern Hoermander condition}
Let $f \in L^1_{\text{loc}}(\R_+).$
Consider the conditions
\begin{enumerate}
\item $f$ satisfies \eqref{Equ Classical Hoermander condition},
\item $f$ satisfies
\begin{equation}\label{Equ Modern Hoermander condition}
\sup_{n \in \Z} \| \equi_n f_e \|_{\Sob^\alpha} < \infty,
\end{equation}
where $(\equi_n)_{n \in \Z}$ is an equidistant partition of unity.
By this we mean the following:
Let $\equi \in C^\infty_c(\R).$
Assume that $\supp \equi \subset [-1,1]$ and $\sum_{n=-\infty}^\infty \equi(t-n) = 1$ for all $t \in \R.$
For $n \in \Z,$ we put $\equi_n = \equi(\cdot - n)$ and call $(\equi_n)_{n \in \Z}$ an equidistant partition of unity.
One easily checks that \eqref{Equ Modern Hoermander condition} does not depend on the particular choice of $(\equi_n)_{n \in \Z}.$
\item $\|f\|_{\Ha} < \infty.$
\end{enumerate}
Then $(1) \Rightarrow (2)$ if $\alpha_1 \geq \alpha$ and $(2) \Rightarrow (1)$ if $\alpha \geq \alpha_1.$
Further, $(2) \Leftrightarrow (3).$
\end{lem}

\subsection{Functional calculus for $0$-sectorial operators}\label{Subsec Prelims Functional calculus}

In order to reduce the H\"ormander calculus to the $\HI$ calculus we will use the following approximation of $\Sobexp^\beta_p$ and $\Hor^\beta_p$ functions holomorphic in a sector.
The following lemma which is proved in \cite[Lemma 4.15]{Kr} will be useful.

\begin{lem}\label{Lem HI dense in diverse spaces}
Let $p \in [1,\infty)$ and $\beta > \frac1p.$
Then $\bigcap_{0 < \omega < \pi}\HI(\Sigma_\omega) \cap \Sobexp^\beta_p$ is dense in $\Sobexp^\beta_p.$
More precisely, if $f \in \Sobexp^\beta_p,$ $\psi \in C^\infty_c$ such that $\psi(t) =1$ for $|t| \leq 1$ and $\psi_n = \psi(2^{-n}(\cdot)),$
then
\[ (f_e \ast \check\psi_n) \circ \log \in \bigcap_{0 < \omega < \pi} \HI(\Sigma_\omega) \cap \Sobexp^\beta_p \text{ and }(f_e \ast \check\psi_n) \circ \log \to f\text{ in }\Sobexp^\beta_p. \]
Thus if $f$ happens to belong to several $\Sobexp^\beta_p$ spaces as above with different indices,
then it can be simultaneously approximated by a holomorphic sequence in any of these spaces.
\end{lem}

Lemma \ref{Lem HI dense in diverse spaces} enables to base the $\Sobexp^\beta_p$ calculus on the $\HI$ calculus.

\begin{defi}\label{Def Line calculi}
Let $A$ be a 0-sectorial operator, $p \in [1,\infty)$ and $\beta > \frac1p.$
We say that $A$ has a (bounded) $\Sobexp^\beta_p$ calculus if there exists a constant $C > 0$ such that
\[
\|f(A)\| \leq C \|f\|_{\Sobexp^\beta_p}\quad (f \in \bigcap_{0 < \omega < \pi} \HI(\Sigma_\omega) \cap \Sobexp^\beta_p).
\]
In this case, by the density of $\bigcap_{0 < \omega < \pi} \HI(\Sigma_\omega) \cap \Sobexp^\beta_p$ in $\Sobexp^\beta_p,$ the algebra homomorphism $u : \bigcap_{0 < \omega < \pi} \HI(\Sigma_\omega) \cap \Sobexp^\beta_p \to B(X)$ given by $u(f) = f(A)$ can be continuously extended in a unique way to a bounded algebra homomorphism
\[u: \Sobexp^\beta_p \to B(X),\,f \mapsto u(f).\]
We write again $f(A)=u(f)$ for any $f \in \Sobexp^\beta_p.$
Assume that $A$ has a $\Sobexp^\beta_p$ calculus and a $\Sobexp^{\beta'}_{p'}$ calculus.
Then for $f \in \Sobexp^\beta_p \cap \Sobexp^{\beta'}_{p'},\, f(A)$ is defined twice by the above.
However, Lemma \ref{Lem HI dense in diverse spaces} shows that these definitions coincide.
\end{defi}

\begin{defi}\label{Def R-bdd Matr R-bdd calculus}
Let $A$ be a $0$-sectorial operator.
We say that $A$ has an $R$-bounded $\Sea$ calculus if $A$ has a $\Sea$ calculus, which is an $R$-bounded mapping in the sense of \cite[Definition 2.7]{KrLM}, i.e.
\[ R(\{ f(A):\: \|f\|_{\Sea} \leq 1\}) < \infty.\]
\end{defi}

\begin{defi}\label{Def Hor calculus}
Let $p \in [1,\infty),\,\alpha > \frac1p$ and let $A$ be a $0$-sectorial operator.
We say that $A$ has a (bounded) $\Ha$ calculus if there exists a constant $C > 0$ such that
\begin{equation}\label{Equ Def Hor sect calculus}
\|f(A)\| \leq C \|f\|_{\Ha} \quad (f \in \bigcap_{\omega \in (0,\pi)} \HI(\Sigma_\omega) \cap \Ha).
\end{equation}
Similarly as in Definition \ref{Def R-bdd Matr R-bdd calculus}, we say that $A$ has an $R$-bounded $\Ha$ calculus if moreover $R(\{f(A):\: f \in \bigcap_{\omega \in (0,\pi)} H^\infty(\Sigma_\omega) \cap \Ha , \: \|f\|_{\Ha} \leq 1\}) < \infty.$
\end{defi}


Finally, we record some norm estimates which will be useful later.
The functions in the following norm estimates correspond via functional calculus to typical operator families.
We use the short hand notation $\langle t \rangle = \sqrt{ 1 + t^2 }.$

\begin{lem}\label{Lem Mihlin norms of functions}
Let $p \in [1,\infty)$ and $\alpha > \frac1p.$
We have the following $\Ha$ norm estimates for functions depending on the variable $\lambda > 0.$
\begin{enumerate}
\item For $\theta \in (-\pi,\pi),\: \| \exp(-e^{i\theta} \lambda) \|_{\Ha} \lesssim (\frac{\pi}{2} - |\theta|)^{-\alpha}.$
\item For $s \in \R,\: \| (1 + |s|\lambda)^{-\alpha} \exp(is\lambda) \|_{\Ha} \lesssim 1.$
\item For $t \in \R,\: \|\lambda^{it}\|_{\Hor^{\alpha-\epsilon}_p} \lesssim \ta.$
\item For $1<\beta<\alpha,$ $\sup_{u > 0} \|(1 - \lambda/u)^{\alpha - 1}_+\|_{\Hor^\beta_1} < \infty,$ where $x_+ = \max(x,0).$
\end{enumerate}
\end{lem}

\begin{proof}
(1) Let first $\alpha = n \in \N_0$ and put $f(\lambda) = \exp(-e^{i\theta} \lambda).$
We have by Lemma \ref{Lem Classical and modern Hoermander condition}, $\| f \|_{\Hor^n_p} \lesssim \sup\{ |\lambda^k f^{(k)}(\lambda)| :\: k = 0,\ldots,n,\, \lambda > 0 \}.$
The latter can easily be estimated by $(\frac{\pi}{2} - |\theta|)^{-n}.$
For non-integer $\alpha = n + \vartheta,\, \vartheta \in (0,1),$ we use the complex interpolation $[W^n_p, W^{n+1}_p]_{\vartheta} = W^\alpha_p$ to deduce by Lemma \ref{Lem Classical and modern Hoermander condition}
\begin{align*}
\| f \|_{\Ha} & \cong \sup_{t > 0} \|f(t\cdot)\psi\|_{\Sobolev^\alpha_p} \leq \sup_{t > 0} \|f(t\cdot) \psi\|_{\Sobolev^n_p}^{1- \vartheta} \cdot \|f(t\cdot) \psi\|_{\Sobolev^{n+1}_p}^{\vartheta} \\
& \lesssim (\frac{\pi}{2} - |\theta|)^{-[n (1 - \vartheta) + (n+1) \vartheta ]} = (\frac{\pi}{2} - |\theta|)^{-\alpha}.
\end{align*}
(2) For $\Re z > 0,$ put $g(z;\lambda) = (1 + |s|\lambda)^{-z} \exp(is\lambda).$
Similarly to (1), it is easy to check that $\|g(n+i\tau;\cdot)\|_{\Hor^n_p} \lesssim \langle \tau \rangle^n$ and $\sup_{\lambda >0;\: k=0,1,\ldots,n} \lambda^k \frac{d^k}{d\lambda^k} |g(n+i\tau;\lambda)| \lesssim \langle \tau \rangle^n$ for $\tau \in \R.$
Then we can deduce from Stein's subexponential complex interpolation \cite[Theorem 1]{Stei} that $\|g(\alpha;\cdot)\|_{\Ha} \lesssim 1$ for general $\alpha.$\\

\noindent
(3) For integer $\alpha \geq 1$ and any $p \in (1,\infty),$ it is easy to check that $\|\lambda \mapsto \lambda^{it} \|_{\Hor^\alpha_p} \lesssim \langle t \rangle^\alpha.$
Likewise, since $\| \lambda \mapsto \lambda^{it} \|_{L^\infty(\R_+)} \lesssim 1,$ we have $\|\lambda \mapsto \lambda^{it} \|_{\Hor^{\frac1q + \epsilon}_q} \lesssim 1$ for any $q \in (1,\infty)$ and $\epsilon > 0.$
Now for $\alpha > 1,$ apply complex interpolation similar to parts (1) and (2), between the cases $\alpha_1 = \lfloor \alpha \rfloor$ and $\alpha_2 = \alpha_1 + 1,$ and for $\alpha < 1,$ between the $\Hor^{\frac1q + \epsilon}_q$ estimate with $q$ close to $\infty$ and the $\Hor^\alpha_p$ estimate with $\alpha = 1$ and $p = q.$\\

\noindent
(4) By Lemma \ref{Lem Elementary Mih Hor} (3), 
$\|(1 - \lambda/u)^{\alpha - 1}_+\|_{\Hor^\beta_1} \cong \|(1-\lambda)^{\alpha-1}_+\|_{\Hor^\beta_1},$ so that it only remains to prove that $(1 - \lambda)^{\alpha-1}_+$ belongs to $\Hor^\beta_1.$
It remains to estimate $\sup_{n \in \Z} \| (1 - 2^n \lambda)^{\alpha - 1}_+ \dyad_0(\lambda) \|_{\Sobolev^\beta_1},$ where $\dyad_0 \in C^\infty_c$ such that $\supp \dyad_0 \subset [\frac12,2].$
Let first $n \leq -2$ and $m$ the least integer greater or equal than $\beta.$
We have 
\begin{align*}
\| (1 - 2^n \lambda)^{\alpha - 1}_+ \dyad_0(\lambda) \|_{\Sobolev^\beta_1} & \leq \| (1 - 2^n \lambda)^{\alpha - 1}_+ \dyad_0(\lambda) \|_{\Sobolev^m_1} \lesssim \max_{k=0,1,\ldots,m} \int_{\frac12}^2 | \frac{d^k}{d \lambda^k} (1 -  2^n \lambda)^{\alpha - 1}| d\lambda \\
& \lesssim \max_{k=0,1,\ldots,m} \int_{\frac12}^2 2^{nk} (1 - 2^n \lambda)^{\alpha - 1 - k} d \lambda.
\end{align*}
Since $n \leq -2,$ we have $\frac12 \leq (1 - 2^n \lambda) \leq 1$ for $\frac12 \leq \lambda \leq 2,$ and $2^{nk} \leq 1,$ so that the above expression is uniformly bounded in $n \leq -2.$
For $n \geq 1,$ we have $(1 - 2^n \lambda)^{\alpha-1}_+ = 0$ for $\lambda \geq \frac12 \geq 2^{-n},$ so that $(1 - 2^n \lambda)^{\alpha - 1}_+ \dyad_0(\lambda) \equiv 0.$
Finally for $n \in \{-1,0\},$ $\| (1 - 2^n \lambda)^{\alpha - 1}_+ \dyad_0(\lambda) \|_{\Sobolev^\beta_1} = \| (1 - 2^{2n} \lambda^2)^{\alpha - 1}_+ \left[ ( 1 + 2^n \lambda )^{1 - \alpha}_+ \dyad_0(\lambda) \right] \|_{\Sobolev^\beta_1}.$
According to \cite[p.~11]{COSY}, one has $(1 - \lambda^2)^{\alpha - 1}_+ \in \Sobolev^\beta_1 \Longleftrightarrow \beta < \alpha,$
and the space $\Sobolev^\beta_1$ is invariant under dilations $f \mapsto f(t\cdot),$ so that the first factor
$(1 - 2^{2n} \lambda^2)_+^{\alpha - 1}$ in the last expression belongs to $\Sobolev^\beta_1.$
Furthermore, the second expression in between the brackets belongs to $C^\infty_c,$ so that the whole term belongs to $\Sobolev^\beta_1.$
The lemma is proven.
\end{proof}

\section{Extended H\"ormander calculus}
\label{sec-extended-Hoermander-calculus}

In statements like \ref{W_alpha}, \ref{BIP_alpha} and \ref{BR_alpha} in the introduction, one would like to give a meaning to ``operators'' $f(A)$ such as $\exp(isA), A^{it}$ and $(1- A/u)_+^\alpha$ before we have established the boundedness of a $\Hor^\alpha_p$ calculus.
In the classical case of a selfadjoint operator on $L^2(U)$ and $X = L^q(U),$ one considers $f(A)$ as defined by functional calculus on $L^2$ and thinks of $f(A)$ on $L^q \cap L^2$ as the part of $f(A)$ on $L^q \cap L^2.$
In the formulas it is then implicitly assumed that the unbounded operator ``$f(A)$'' has a continuous extension to an operator in $B(L^q(U)).$
This assumption is then part of \ref{W_alpha}, \ref{BIP_alpha} and \ref{BR_alpha}.
In the general case of $0$-sectorial operators on a Banach space one has to say unfortunately a little bit more to circumvent this purely formal difficulty.
It is convenient to consider the subspaces $D(\theta) = D(A^\theta) \cap R(A^\theta)$ for $\theta > 0,$
which are Banach spaces with norm $\|x\|_{D(\theta)} = \| \rho^{-\theta}(A)x \|_X$ and the $D(\theta)$ form a decreasing sequence of spaces when $\theta$ grows.
Here $\rho(\lambda) = \lambda ( 1 + \lambda)^{-2}$ belongs together with its powers $\rho^\theta$ to $\HI_0(\Sigma_\omega)$ for any $\omega \in (0,\pi),$ and $R(\rho^\theta(A)) = D(A^\theta) \cap R(A^\theta).$
Note that $D(\theta)$ is dense in $X$ (see \cite[9.4 Proposition (c)]{KuWe} for the case $\theta \in \N$).
Then to make sense of $f(A)$ we will use an ``extended'' version of the the $\HI$ and $\Ha$ calculus which only produces closable operators on $D(\theta).$

For $\omega \in (0,\pi),$ define the algebras of functions $\Hol(\Sigma_\omega) = \{ f : \Sigma_\omega \to \C :\: \exists \: n \in \N :\: \rho^n f \in \HI(\Sigma_\omega) \}.$
For a proof of the following lemma, we refer to \cite[Section 15B]{KuWe} and \cite[p.~91-96]{Haasa}.

\begin{lem}\label{Lem Hol}
Let $A$ be a $0$-sectorial operator.
There exists a linear mapping, called the extended holomorphic calculus,
\begin{equation}\label{Equ Extended HI calculus}
\bigcup_{\omega > 0} \Hol(\Sigma_\omega) \to \{ \text{closed and densely defined operators on }X \},\: f \mapsto f(A)
\end{equation}
extending \eqref{Equ Cauchy Integral Formula} such that for any $f,g \in \Hol(\Sigma_\omega),$ $f(A)g(A)x = (f g)(A)x$ for $x \in \{y \in D(g(A)):\: g(A)y \in D(f(A)) \} \subset D((fg)(A))$ and
$D(f(A)) = \{ x \in X :\: (\rho^n f)(A) x \in D(A^n) \cap R(A^n) = D(n) \},$ where $(\rho^n f)(A)$ is given by \eqref{Equ Cauchy Integral Formula}, i.e. $n \in \N$ is sufficiently large.
\end{lem}

In an analogous way we introduce an extended $\Ha$ calculus $\Phi_A : \Ha \to B(D(\theta),X)$ for some $\alpha > \frac1p,\: \theta > 0.$
The existence of such a calculus is known in many concrete situations.

For Lemma \ref{Lem auxiliary calculus multiplicative} and the sequel, we need the following notion.

\begin{defi}
Let $\dyad \in C^\infty_c(\R_+).$
Assume that $\supp \dyad \subset [\frac12,2]$ and $\sum_{n=-\infty}^\infty \dyad(2^{-n} t) =1$ for all $t > 0.$
For $n \in \Z,$ we put $\dyad_n = \dyad(2^{-n} \cdot)$ and call $(\dyad_n)_{n \in \Z}$ a dyadic partition of unity.
For the existence of such a partition, we refer to the idea in \cite[Lemma 6.1.7]{BeL}.
\end{defi}

\begin{lem}
\label{Lem auxiliary calculus multiplicative}
Let $A$ be a $0$-sectorial operator on a Banach space $X.$
Assume that one of the following conditions holds.
\begin{enumerate}
\item $\|A^{it}\|$ is polynomially bounded in $t \in \R.$
\item $\| \exp(-zA) \| \leq C \left( \frac{|z|}{\Re z} \right)^\beta$ for some $\beta > 0$ and all $z \in \C_+.$
\item $\| \lambda R(\lambda,A) \| \leq C | \arg(\lambda) |^{-\beta}$ for some $\beta > 0$ and all $\lambda \in \C \backslash [ 0, \infty ).$
\end{enumerate}
Then for $p \in [1,2]$ and some suitable $\alpha > \frac1p$ and $\theta > 0,$ there exists an auxiliary functional calculus $\Phi_A : \Ha \to B(D(\theta),X),$ which is a linear mapping and has the compatibility that for $\omega \in (0,\pi),$ $f \in \HI(\Sigma_\omega) \cap \Ha$ and $x \in D(\theta),$ there holds $\Phi_A(f)x = f(A)x,$ where the right hand side is defined by the holomorphic functional calculus from Lemma \ref{Lem Hol}.
Moreover, $\Phi_A(f)$ is a closable operator over $X.$
We can denote without ambiguity $f(A)$ its closure.
Then we have the further compatibilities
\begin{enumerate}
\item If $f \in \Ha ,\: g \in \HI(\Sigma_\omega)$ and $x \in D(\theta),$ then $f(A) g(A)x = (fg)(A)x.$
\item If $f \in \Ha,$ $g \in \HI_0(\Sigma_\omega)$ and $x \in D(\theta),$ then $g(A) f(A)x = (gf)(A) x.$
\item If $f,g \in \Ha$ and $x \in D(\theta),$ then $g(A)x \in D(f(A))$ and $f(A)g(A)x = (fg)(A)x.$
\end{enumerate}
\end{lem}

\begin{proof}
The Lemma is proved for the case $p = 2$ and a mapping $\Phi_A^{\Sea} : \Sea \to B(D(\theta),X),$ for the imaginary powers in \cite{KrW2}, for the resolvents in \cite[Proposition 3.9]{KrW1}, and for the semigroup, see this paper at Remark \ref{rem-semigroup-auxiliary-calculus}.
The case of general $p$ is entirely similar.
To pass from $\Sea$ functions to $\Ha$ functions, note that for $f \in \Ha$ and $\nu > 0,$ we have
$f \rho^\nu \in \Sea.$
Indeed, with $(\dyad_k)_{k \in \Z}$ a dyadic partition of unity and $\widetilde{\dyad_k} = \dyad_{k-1} + \dyad_k + \dyad_{k+1},$ we have
\begin{align*}
\| f \rho^\nu \|_{\Sea} & \leq \sum_{k \in \Z} \| f \rho^\nu \dyad_k \|_{\Sea} \\
& \lesssim \sum_{k \in \Z} \| f \dyad_k \|_{\Sea} \| \rho^{\nu} \widetilde{\dyad_k} \|_{\Sea} \\
& \lesssim \| f \|_{\Ha} \sum_{k \in \Z} 2^{-|k| \nu} \lesssim \|f \|_{\Ha}.
\end{align*}
Now suppose that $\Phi_A^{\Sea} : \Sea \to B(D(\theta),X)$ is an auxiliary calculus as above.
Let $\theta' > \theta$ and set $\nu = \theta' - \theta > 0.$
Then $\Phi_A : \Ha \to B(D(\theta'),X),\: f \mapsto (x = \rho^\nu(A) y \mapsto \Phi_A^{\Sea}(f \rho^\nu)y)$ is the desired auxiliary calculus of the lemma, where $y \in D(\theta).$
Now it is easy to check that the already established compatibilities of $\Phi_A^{\Sea}$ carry over to $\Phi_A.$
\end{proof}

In some cases, it is convenient to consider an even smaller domain of definition than $D(\theta):$
For a $0$-sectorial operator $A$ with auxiliary functional calculus $\Phi_A$ as above, we define the following subset $D_A$ of $X.$
Let $(\dyad_n)_{n \in \Z}$ be a dyadic partition of unity.
\begin{equation}
\label{Equ D}
D_A = \left\{\sum_{n = -N}^N \dyad_n(A) x:\: N \in \N,\: x \in D(\theta) \right\}.
\end{equation}
We call $D_A$ the calculus core of $A.$
According to \cite{KrW2}, $D_A$ is dense in $X.$

As for the $\HI$ calculus, there is an extended $\Ha$ calculus which is defined for $f : (0,\infty) \to C$ with $f \rho^\nu = f(\cdot) (\cdot)^\nu ( 1 + (\cdot))^{-2\nu} \in \Ha$ for some $\nu > 0,$ as a counterpart of \eqref{Equ Extended HI calculus}.

\begin{defi}\label{Def Unbounded Besov or Sobolev calculus}
Let $A$ have an auxiliary calculus $\Phi_A:\Ha \to B(D(\theta),X).$
Let $f : (0,\infty) \to \C$ with $f \rho^\nu \in \Ha$ for some $\nu > 0.$
We define the operator $f(A)$ on $D_A$ by
\[ f(A) (\sum_{n = -N}^N \dyad_n(A)x) = \sum_{n = -N}^N (f \dyad_n)(A)x.\]
Note that this definition does not depend on the representation $\sum_{n = -N}^N \dyad_n(A) x$ of the element in $D_A.$
\end{defi}

\begin{lem}\label{Lem Soaloc calculus}
Let $1 \leq p < \infty$ and $A$ have an auxiliary calculus $\Phi_A: \Ha \to B(D(\theta),X)$ such that $\Phi_A(f)x = f(A)x$ for $x \in D(\theta)$ and $f \in \HI(\Sigma_\omega) \cap \Ha.$
Assume that $f \rho^\nu \in \Ha$ for some $\nu > 0,$ and $g$ a further function, where we suppose the same assumptions as $f.$
\begin{enumerate}
\item[(a)] The operator $f(A)$ is closable.
We denote the closure by slight abuse of notation again by $f(A).$
\item[(b)] If furthermore $f \in \Ha$ then $f(A)$ coincides with the operator defined by the calculus $\Phi_A.$
If $f \in \Hol(\Sigma_\omega)$ for some $\omega \in (0, \pi),$ then
$f(A)$ coincides with the (unbounded) holomorphic calculus of $A.$
\item[(c)] For any $x \in D_A,$ we have $g(A)x \in D(f(A))$ and $f(A)g(A)x = (fg)(A)x.$
\end{enumerate}
\end{lem}

\begin{proof}
This is proved in \cite{KrW2} in the case $p = 2$ and for $\Sea$ in place of $\Ha.$
Then the case $p \in [1,2]$ is entirely similar, and for the $\Ha$ case, we can proceed as in the proof of Lemma \ref{Lem auxiliary calculus multiplicative} noting that $f \rho^\nu \in \Ha$ implies $f \rho^{\nu'} \in \Sea$ for $\nu' > \nu > 0.$
\end{proof}

Assume that $A$ has a $\Ha$ calculus in the sense of Definition \ref{Def Hor calculus}.
Then it has an auxiliary functional calculus as in Lemma \ref{Lem Soaloc calculus} and the estimate 
\eqref{Equ Def Hor sect calculus} extends to all of $f \in \Ha.$

The following lemma gives several representation formulas for the $\Sea$ calculus and the auxiliary $\Ha$ calculus, in terms of the $C_0$-group $A^{it},$ the wave group $e^{itA}$ of (in general) unbounded operators and the Bochner-Riesz means $(1 - A/u)^{\alpha - 1}_+.$

\begin{lem}\label{Lem Sobolev calculus}
Let $X$ be a Banach space with dual $X',$ and let $p \in [1,2].$
Let $\alpha > \frac1p,$ so that $\Soa$ is a Banach algebra.
Let $A$ be a $0$-sectorial operator with imaginary powers $A^{it}.$
\begin{enumerate}
\item
Assume that for some $C > 0$ and all $x \in X,\,x'\in X'$
\begin{equation}\label{Equ group weak L2}
\left( \int_\R |\tma \spr{A^{it}x}{x'}|^p dt \right)^{1/p} \leq C \| x \| \, \| x' \|.
\end{equation}
Then $A$ has a bounded $\Sea$ calculus.
Moreover, for any $f \in \Sea,$ $f(A)$ is given by
\[ \spr{f(A) x}{x'} = \frac{1}{2\pi} \int_\R (f_e)\hat{\phantom{i}}(t) \spr{A^{it}x}{x'} dt \quad (x \in X,\: x' \in X'). \]
The above integral exists as a strong integral if moreover $\tma A^{it} x \in L^p(\R;X).$
\item Conversely, if $p = 2$ and $A$ has a $\Sobexp^\alpha_2$ calculus, then \eqref{Equ group weak L2} holds.
\item Assume that $A$ has an auxiliary calculus $\Phi_A : \Hor^\gamma_p \to B(D(\theta),X)$ for some $\theta \geq 0$ and $\gamma > \frac1p \geq \frac12.$
Then for $x$ belonging to the calculus core $D_A$ of $A$ and for $f \in C^\infty_c(0,\infty),$
\[ f(A) x = \frac{1}{2\pi} \int_\R \hat{f}(t) e^{itA} x dt. \]
\item Assume that $1 < \beta < \alpha$ and that $A$ has an auxiliary calculus $\Phi_A : \Hor^\beta_1 \to B(D(\theta),X)$ for some $\theta \geq 0$ and an $\HI(\Sigma_\omega)$ calculus for some $\omega \in (0,\frac{\pi}{2}).$
Then the Bochner-Riesz means $R^{\alpha - 1}_u(A)$ are densely defined closed operators, where $R^{\alpha - 1}_u(\lambda) = (1 - \lambda/u)_+^{\alpha - 1} \in \Hor^\beta_1$ and $t_+ = \max(t,0).$
Moreover, for $f \in C^\infty_c(\R_+)$ with $\supp f \subset [\frac12,2]$ and $x \in D_A,$ one has
\[ f(A) x = \frac{(-1)^m}{\Gamma(\alpha)} \int_0^\infty f^{(\alpha)}(u) u^{\alpha - 1} R^{\alpha - 1}_u(A) x du,\]
where $f^{(\alpha)}$ is defined e.g. by $\hat{f^{(\alpha)}}(\xi) = (-i \xi)^\alpha \hat{f}(\xi),\: \xi \in \R,$ and $m = \lfloor \alpha \rfloor.$

The same formula holds if $X = L^p(U,\mu)$ for some $1 < p < \infty,$ $A$ is self-adjoint on $L^2(U,\mu),$ and $R^{\alpha-1}_u(A)$ is densely defined on $L^p$ by the self-adjoint spectral calculus for $x \in L^2 \cap L^p.$
\end{enumerate}
\end{lem}

\begin{proof}
(1) and (2) This can be proved with the Cauchy integral formula \eqref{Equ Cauchy Integral Formula} in combination with the Fourier inversion formula, see \cite[Proposition 4.22]{Kr}.\\

\noindent
(3) Let $\phi \in C^\infty_c(0,\infty).$
By the Fourier inversion formula, one has
\begin{equation}\label{Equ Proof Fourier inversion Formula}
f(\lambda) \phi(\lambda) = \frac{1}{2\pi} \int_\R \hat{f}(t) e^{it\lambda} \phi(\lambda) dt,
\end{equation}
which is in particular a Bochner integral formula with values in the space $\Hor^\gamma_p,$ since $\|\lambda \mapsto e^{it\lambda} \phi(\lambda)\|_{\Sobexp^\gamma_p} \lesssim \langle t \rangle^\gamma$ and $\hat{f}(t)$ is rapidly decreasing.
If $x = \sum_{n = - N}^N \dyad_n(A) y \in D_A,$ then let $\phi = \sum_{n = -N}^N \dyad_n.$
Applying the auxiliary $\Hor^\gamma_p$ calculus on both sides of \eqref{Equ Proof Fourier inversion Formula} yields (3).\\

\noindent
(4) For $f$ as above and $\phi \in C^\infty_c(0,\infty)$, one has the representation formula
\begin{equation}\label{Equ Bochner-Riesz}
f(s) \phi(s) = \frac{(-1)^m}{\Gamma(\alpha)} \int_s^\infty (t-s)^{\alpha - 1} f^{(\alpha)}(t) \phi(s) dt = \frac{(-1)^m}{\Gamma(\alpha)} \int_0^\infty (t-s)_+^{\alpha - 1} f^{(\alpha)}(t) \phi(s) dt, \: s > 0,
\end{equation}
where $m = \lfloor \alpha \rfloor$ \cite[p.~1011]{GaMi}.
This implies that if $f(s) = 0$ for $s \geq r,$ then $f^{(\alpha)}(s) = 0$ for $s \geq r.$
It is now easy to check that the integral formula \eqref{Equ Bochner-Riesz} holds as a Bochner integral in $\Hor^\beta_1,$ and applying the auxiliary $\Hor^\beta_1$ calculus yields with $\phi \in C^\infty_c(0,\infty)$ such that $x = \phi(A)y,$ $y \in D(\theta),$
\[ f(A)x = (f\phi)(A)y = \frac{(-1)^m}{\Gamma(\alpha)} \int_0^\infty f^{(\alpha)}(t) (t-A)_+^{\alpha - 1} \phi(A) y dt =
\frac{(-1)^m}{\Gamma(\alpha)} \int_0^\infty f^{(\alpha)}(u) u^{\alpha - 1} R_u^{\alpha - 1}(A)x dt.\]
The case $X = L^p(U,\mu),$ $A$ self-adjoint on $L^2$ and $x \in L^p(U) \cap L^2(U)$ also follows from \eqref{Equ Bochner-Riesz}, regarded as a Bochner integral in $\mathcal{B}^\infty(\R_+)$ by applying the self-adjoint calculus.
\end{proof}

\section{The Localization Principle and $\Ha$ calculus}\label{Sec 5 Mihlin}


In this section we reduce the problem of showing the boundedness of the $\Ha$ calculus to boundedness on functions with compact support, by a localization principle.
We are going to show that in the presence of some $\HI$ calculus, an $R$-bounded $\Sea$ calculus can be improved to an ($R$-bounded) $\Ha$ calculus.
This replaces the Littlewood-Paley arguments in the proof of classical spectral multiplier theorems.
Indeed, a somewhat weaker assumption than $R$-bounded $\Sea$ calculus suffices, as is shown in the following theorem.

\begin{thm}\label{Thm Localization Principle}
Let $A$ be a $0$-sectorial operator having an $\HI(\Sigma_\sigma)$ calculus for some $\sigma \in (0,\pi)$ and assume that $X$ has property $(\alpha).$
Consider the following conditions for $p \in [1,\infty)$ and $\alpha > \frac1p.$
\begin{enumerate}
\item $A$ has an $R$-bounded $\Sea$ calculus.
\item $A$ has an auxiliary calculus $\Phi_A : \Hor^\beta_p \to B(D(\theta),X)$ for some $\beta,\theta > 0,$ (so that $f(2^nA)$ is well-defined for $f \in C^\infty_c(\R_+)$ and $n \in \Z$), and we have
\begin{equation}\label{Equ Localized R-bound}
\left\{ f(2^n A):\: f \in C^\infty_c(\R_+),\: \supp f \subset [\frac12, 2],\: \|f\|_{\Sobolev^\alpha_p} \leq 1,\: n \in \Z \right\} \text{ is }R\text{-bounded.}
\end{equation}
\item $A$ has an $R$-bounded $\Ha$ calculus.
\end{enumerate}
Then all these conditions are equivalent.
If $X$ does not have property $(\alpha),$ but $A$ is $R$-sectorial, then one still has (1) $\Longrightarrow$ (2) $\Longrightarrow$ (3'), with (3'): $A$ has a bounded $\Ha$ calculus.
\end{thm}

\begin{proof}
Since $\Sea \subset \Ha,$ we clearly have that (3) implies (1).
Furthermore, if $f \in C^\infty_c(\R_+)$ with $\|f\|_{\Sobolev^\alpha_p} \leq 1,\,\supp f \subset [\frac12,2], \: n \in \Z$ and $g = f(2^n \cdot),$ then $\|g\|_{\Sea} \lesssim 1.$
Indeed, by the fixed support of $f,$ the Sobolev norms of $f(e^{(\cdot)+n \log(2)})$ and $f((\cdot) + n \log(2))$ are equivalent, and we thus have
\begin{align*}
\|g\|_{\Sea} & = \|f(2^n \cdot)\|_{\Sea} = \|f(2^n e^{(\cdot)})\|_{\Sobolev^\alpha_p} \\
& = \| f(e^{(\cdot) + n \log(2)}) \|_{\Sobolev^\alpha_p} \cong \|f((\cdot) + n \log(2))\|_{\Sobolev^\alpha_p} = \|f\|_{\Sobolev^\alpha_p} \leq 1.
\end{align*}
Thus, (1) implies (2).
It remains to show that (2) implies (3).
Consider a function $\phi \in \HI_0(\Sigma_\nu)$ such that $\sum_{n \in \Z}\phi^3(2^{-n} \lambda) = 1$ for any $\lambda \in \Sigma_\nu$ and some $\nu > \sigma.$
Furthermore, consider a function $\eta \in C^\infty_c$ with $\supp \eta \subset [\frac12,2]$ such that $\sum_{n \in \Z} \eta(2^{-n} t) = 1$ for any $t > 0.$
Let $f_1,\ldots,f_N \in C^\infty(\R_+)$ with $\|f_j\|_{\Ha} \leq 1$ for $j = 1,\ldots,N.$
Then for $x_j \in D_A,$ the calculus core,
\begin{align*}
\E \| \sum_{j=1}^N \epsilon_j f_j(A) x_j \| & = \E \| \sum_{j=1}^N \sum_{k \in \Z} \epsilon_j \eta(2^{-k}A)f_j(A)x_j \| \\
& \cong \E \E' \| \sum_{j =1}^N \sum_{n,k \in \Z} \epsilon_j \epsilon_n' \phi^2(2^{-n}A) \eta(2^{-k}A)f_j(A) x_j \| \\
& \leq \sum_{l \in \Z} \E \E' \| \sum_{j=1}^N \sum_{n \in \Z} \epsilon_j \epsilon_n' [\phi(2^{-n}A) \eta(2^{-n-l}(A) f_j(A)] \phi(2^{-n} A) x_j\| \\
& \leq \left( \sum_{l \in \Z} C_l \right) \E \E' \| \sum_{j=1}^N \sum_{n \in \Z} \epsilon_j \epsilon_n' \phi(2^{-n} A)x_j \| \\
& \lesssim  \left( \sum_{l \in \Z} C_l \right) \E \|\sum_{j=1}^N \epsilon_j x_j \|,
\end{align*}
where we have used that $\|x\| \cong \E \| \sum_{n \in \Z} \epsilon_n \phi^2(2^{-n}A) x \| \cong \E \| \sum_{n \in \Z} \epsilon_n \phi(2^{-n}A) x \|.$
Indeed, the second expression is estimated by the third one, since $\{ \phi(2^{-n}A):\:n \in \Z \}$ is $R$-bounded by the $R$-boundedness of the $\HI(\Sigma_\nu)$ calculus \cite[12.8 Theorem]{KuWe}.
The third expression is estimated by the first one according to \cite[12.2 Theorem and 12.3 Remark]{KuWe}.
Finally the first expression is estimated by the second one again by \cite[12.2 Theorem and 12.3 Remark]{KuWe} and $|\langle x, x' \rangle| = | \E \langle \sum_{n \in \Z} \epsilon_n \phi^2(2^{-n}A)x, \sum_{k \in \Z} \epsilon_k \phi(2^{-k} A)x' \rangle| \leq \E \|\sum_n \epsilon_n \phi^2(2^{-n}A) x\| \: \E \| \sum_k \epsilon_k \phi(2^{-k}A)'x'\| \lesssim \E \|\sum_n \epsilon_n \phi^2(2^{-n}A) x\| \: \|x'\|. $

Furthermore, we used property $(\alpha)$ in the fourth line, and $C_l = R(\{ \phi(2^{-n}A)\eta(2^{-n-l}A)f_j(A) :\: n \in \Z,\: j = 1,\ldots,N \})$ and 
\begin{align*}
C_l & \lesssim \sup_{j=1,\ldots,N}\sup_{n \in \Z} \|\phi(2^{l}\cdot)\eta f_j(2^{n+l}\cdot)\|_{\Sea} \\
& \lesssim \sup_{j=1,\ldots,N}\sup_{k \in \Z} \|\eta f_j(2^k\cdot)\|_{\Sea} \sup_{m=0,\ldots,\lfloor \alpha \rfloor + 1} \sup_{t \in [\frac12,2]} t^m | \frac{d^m}{dt^m}\phi(2^l \cdot)(t)| \\
& \lesssim \sup_{j=1,\ldots,N}\|f_j\|_{\Ha}  2^{-\epsilon |l|}, \\
& \leq 2^{-\epsilon |l|}
\end{align*}
where $\epsilon > 0$ and we used the fact that $\phi \in \HI_0(\Sigma_\nu).$
Hence $\sum_{l\in\Z} C_l \lesssim \sup_{j=1,\ldots,N}\|f_j\|_{\Ha} < \infty.$
We have shown that
\begin{equation}\label{Equ Proof Localization Principle}
\{ f(A):\: f \in C^\infty(\R_+),\,\|f\|_{\Ha} \leq 1\}
\end{equation}
is $R$-bounded.
In particular, since $C^\infty(\R_+) \supset \bigcap_{\omega > 0}\HI(\Sigma_\omega),$
$A$ has a bounded $\Ha$ calculus in the sense of Definition \ref{Def Hor calculus}, and
by taking the closure of \eqref{Equ Proof Localization Principle}, this calculus is $R$-bounded.

If $X$ does not have property $(\alpha),$ then repeat the proof of (2) $\Longrightarrow$ (3) above with a single function $f \in C^\infty(\R_+),\: \|f\|_{\Ha} \leq 1$ to get in a similar manner (2) $\Longrightarrow$ (3').
Now $\{\phi(2^{-n}A) :\: n \in \Z\}$ is $R$-bounded since $A$ is $R$-sectorial, as soon as $\phi \in \HI_0(\Sigma_\theta)$ with $\theta > \omega_R(A).$
\end{proof}

From the proof of Theorem \ref{Thm Localization Principle}, we obtain the following.

\begin{cor}
Let $A$ be a $0$-sectorial operator on some Banach space $X.$
Assume that $A$ has an $\HI(\Sigma_\sigma)$ calculus for some $\sigma \in (0,\pi)$ and an auxiliary calculus $\Phi_A : \Ha \to B(D(\theta),X)$ for some $\theta > 0,$ $\alpha > \frac1p$ and $1 \leq p < \infty.$
Let $(\dyad_n)_{n \in \Z}$ be a dyadic partition of unity.
\begin{enumerate}
\item If for any $f \in \Ha$ the set $\{ (\dyad_n f)(A) :\: n \in \Z \}$ is $R$-bounded, then $A$ has a bounded $\Ha$ calculus.
\item If $X$ has property $(\alpha),$ then $A$ has an $R$-bounded $\Ha$ calculus if and only if $\{ (\dyad_n f)(A): \: n \in \Z \}$ is $R$-bounded for any $f \in \Ha.$
\end{enumerate}
\end{cor}

\begin{proof}
For the second part, note that $\| \dyad_n f \|_{\Ha} \lesssim \|\dyad_n\|_{\Ha} \|f\|_{\Ha} \lesssim \|f\|_{\Ha},$ so that the ``only if'' part follows.
To complete the proof, it now suffices to show that if $\{ (\dyad_n f)(A): \: n \in \Z \}$ is $R$-bounded for any $f \in \Ha,$ then condition \eqref{Equ Localized R-bound} from Theorem \ref{Thm Localization Principle} is satisfied.

First note that it follows from the closed graph theorem together with the existence of the auxiliary calculus $\Phi_A$ that
\begin{equation}
\label{equ-1-proof-corollary-localization}
R \left( \left\{ (\dyad_n f)(A) : \: n \in \Z \right\} \right) \leq C \|f\|_{\Ha} .
\end{equation}
We claim that \eqref{equ-1-proof-corollary-localization} implies
\begin{equation}
\label{equ-2-proof-corollary-localization}
R\left( \left\{ (\dyad_n f_n)(A): \: n \in \Z \right\} \right) \leq C' \sup_{n \in \Z} \| \dyad_n f_n \|_{\Sea}.
\end{equation}
From \eqref{equ-2-proof-corollary-localization} it is easy to see that \eqref{Equ Localized R-bound} follows.
To prove the claim, we let $f_n \in \Ha$ and set $f = \sum_{n \in 3 \Z} \dyad_n f_n.$
Then $\|f\|_{\Ha} \lesssim \sup_{n \in \Z} \|\dyad_n f_n\|_{\Sea}.$
Set $\widetilde{\dyad_n} = \dyad_{n-1} + \dyad_n + \dyad_{n+1},$ so that $\widetilde{\dyad_n} \dyad_m = \delta_{nm} \dyad_n$ for any $m \in 3 \Z.$
It is clear that \eqref{equ-1-proof-corollary-localization} implies that also $R(\{ (\widetilde{\dyad_n} f)(A) :\: n \in \Z \} ) \lesssim \|f\|_{\Ha}.$
Then, since $\widetilde{\dyad_n}f = \dyad_n f_n$ for any $n \in 3 \Z,$ we obtain
\begin{align*}
R \left( \left\{ (\dyad_n f_n)(A) : \: n \in 3 \Z \right\} \right) & = R \left( \left\{ (\widetilde{\dyad_n} f)(A):\: n \in 3 \Z \right\} \right) \\
& \leq R \left( \left\{ (\widetilde{\dyad_n} f)(A) :\: n \in \Z \right\} \right) \\
& \lesssim \| f \|_{\Ha} \lesssim \sup_{n \in \Z} \|\dyad_n f_n \|_{\Sea}.
\end{align*}
Now apply the same argument to $f = \sum_{n \in 3 \Z + k} \dyad_n f_n$ with $k = 1,2,$ to deduce \eqref{equ-2-proof-corollary-localization}.
\end{proof}

\section{Imaginary powers and the $\Ha$ calculus}\label{Sec 6 Hormander}

In this section, we establish sufficient conditions for the $\Hor^\alpha_p$ calculus based on polynomial growth estimates of imaginary powers.

\begin{thm}\label{Thm Sufficient conditions BIP}
Let $A$ be a $0$-sectorial and $R$-sectorial operator on $X$ having an $\HI(\Sigma_\sigma)$ calculus for some $\sigma \in (0,\pi).$
Let $r \in (1,2],\: \frac1r > \frac{1}{\type X} - \frac{1}{\cotype X}.$
\begin{enumerate}
\item If $\{(1+|t|)^{-\alpha} A^{it} :\: t \in \R \}$ is semi-$R$-bounded and $X$ has property $(\alpha),$ then $A$ has an $R$-bounded $\Hor^\beta_2$ calculus for $\beta > \alpha + \frac12.$
\item If $\{(1+|t|)^{-\alpha} A^{it} :\: t \in \R\}$ is bounded, then $A$ has a bounded $\Hor^\beta_r$ calculus for $\beta > \alpha + \frac1r.$
If $X$ has in addition property $(\alpha),$ then this calculus is $R$-bounded.
\end{enumerate}
Conversely, if $A$ has an $R$-bounded $\Hor^\beta_p$ functional calculus for some $\beta$ and $p$ such that $\beta > \frac1p,$ then $\{(1+|t|)^{-\alpha} A^{it} :\: t \in \R\}$ is $R$-bounded for any $\alpha > \beta.$
\end{thm}

\begin{rem}
\label{rem-type-Lp}
Recall that an $L^p(U)$ space, $1 < p < \infty,$ has type $\min(p,2)$ and cotype $\max(2,p),$ so that in this case, $\frac1r > \left| \frac12 - \frac1p \right|.$
An $L^p(U)$ space always has property $(\alpha).$
\end{rem}

\begin{proof}[Proof of Theorem \ref{Thm Sufficient conditions BIP}]
For (1) (resp.(2)), by Theorem \ref{Thm Localization Principle}, it suffices to show that $A$ has an $R$-bounded $\Sobexp^\beta_2$ calculus (resp. an $R$-bounded $\Sobexp^\beta_r$ calculus).

We now prove (1).
To this end, we show that 
\begin{equation}
\text{The function }t \mapsto \langle t \rangle^{-\beta} A^{it} x\text{ belongs to }\gamma(\R,X)\text{ with norm }\lesssim \|x\|\text{ for }x \in X.\tag*{$(S_{\BIP})_\beta$}
\end{equation}
Here, $\gamma(\R,X)$ is the Gaussian space as defined for example in \cite[Definition 3.7]{vN}.
$(S_{\BIP})_\beta$ will then imply by \cite[Proposition 3.3]{LM2}, see also \cite[Corollary 3.19]{HaKu}, that
\[ \left\{ \int_\R f(t) \langle t \rangle^{-\beta} A^{it} dt :\: \|f\|_{L^2(\R)} \leq 1 \right\} \text{ is }R\text{-bounded.} \]
Since $\|f\|_{\Sobolev^\beta_2} \cong \|\hat{f}(t) \langle t \rangle^{\beta}\|_{L^2},$
this implies by Lemma \ref{Lem Sobolev calculus} that $A$ has a $\Sobexp^\beta_2$ calculus
and that this calculus is $R$-bounded.
Thus, (1) follows from $(S_{\BIP})_\beta,$ which we show now to hold.
By \cite[Proposition 4.1]{VeWe}, we have that 
\begin{align*}
\|\langle t \rangle^{-\beta} A^{it} x\|_{\gamma(\R,X)} & = \| \langle t \rangle^{-(\beta-\alpha)} \langle t \rangle^{-\alpha} A^{it} x \|_\gamma \\
& \leq \| \langle t \rangle^{-(\beta-\alpha)} \|_{L^2(\R)} R_s(\langle t \rangle^{-\alpha} A^{it}:\:t \in \R) \|x\| \\
& \lesssim \|x\|.
\end{align*}
Thus, $(S_{\BIP})_\beta$ follows.

We now prove (2).
Write $U(t) = A^{it}.$
Clearly, $t \mapsto \langle t \rangle^{-\beta} \|U(t)\|$ is dominated by a function in $L^{r}(\R).$
Indeed, $\langle t \rangle^{-\beta} \|U(t)\| = \langle t \rangle^{-(\beta-\alpha)} (\tma \|U(t)\|),$
and the first factor is in $L^r(\R)$ by the choice of $\beta,$ and the second factor is bounded by the assumption (2).
In particular, by Lemma \ref{Lem Sobolev calculus}, $A$ has a $\Sobexp^\beta_r$ calculus, and for $f \in \Sobexp^\beta_r,$ we have
\[f(A)x = \frac{1}{2\pi} \int_\R \langle t \rangle^{\beta} (f_e)\hat{\phantom{i}}(t) \langle t \rangle^{-\beta} U(t) x dt \quad (x\in X).\]
Since $r \leq 2,$ we have by the Hausdorff-Young inequality
$\|\hat{f_e}(t) \langle t \rangle^\beta \|_{L^{r'}(\R)} \lesssim \|f_e\|_{\Sobolev^\beta_r} = \|f\|_{\Sobexp^\beta_r}.$
Further, by the assumption $\frac1r > \frac{1}{\type X}-\frac{1}{\cotype X},$ we can apply Proposition \ref{Prop Hytonen Veraar} and consequently,
\[ R(\{ f(A):\:\|f\|_{\Sobexp^\beta_r} \leq 1\}) \lesssim R(\{ f(A):\: \|\hat{f_e}(t) \langle t \rangle^\beta \|_{L^{r'}(\R)} \leq 1\}) < \infty. \]

The converse statement follows by applying Lemmas \ref{Lem Mihlin norms of functions} and \ref{Lem Elementary Mih Hor}.
\end{proof}

\begin{rem}
The fact that $A$ has a H\"ormander calculus provided that imaginary powers of $A$ grow at most polynomially has been studied before by Meda \cite{Meda}.
He assumes that $-A$ is self-adjoint and generates a contraction semigroup on $L^p$ for all $1 < p < \infty.$
Note that such an operator has an $\HI$ calculus on $L^p$ for any $1 < p < \infty$ \cite{Co}.
If furthermore the imaginary powers satisfy $\| A^{it} \| \leq C_p (1 + |t|)^{\beta |\frac1p-\frac12|}$ for all $1 < p < \infty$ and some $\beta > 0,$ then Meda shows that $A$ has a $\Hor^{\beta/2 + 1 + \epsilon}_\infty$ calculus on $L^p$ for $1 < p < \infty$ \cite[Corollary 1, Theorem 4]{Meda}.
By comparison our result guarantees a $\Hor^{\beta | \frac1p - \frac12| + \frac1r +\epsilon}_r$ calculus with $\frac1r = |\frac1p - \frac12| < 1,$
which is a stronger result according to Lemma \ref{Lem Elementary Mih Hor} (2).
Moreover, our functional calculus is $R$-bounded, not just bounded.
\end{rem}

The optimality of the assumptions of Theorem \ref{Thm Sufficient conditions BIP} will be discussed in Section \ref{Sec Examples Counterexamples}.

\section{Semigroups and the $\Ha$ calculus}
\label{sec-Semigroup}

We have the following analogue of Theorem \ref{Thm Sufficient conditions BIP} considering $R$-bounds on the analytic semigroup of $A$ in place of the imaginary powers.

\begin{thm}\label{Thm Sufficient conditions Sgr}
Let $A$ be a $0$-sectorial operator on $X$ having an $\HI(\Sigma_\sigma)$ calculus for some $\sigma \in (0,\pi).$
Let $r \in (1,2]$ with $\frac1r > \frac{1}{\type X} - \frac{1}{\cotype X}$ and $\alpha \geq 0.$
(See Remark \ref{rem-type-Lp} for the $L^q$-case.)
\begin{enumerate}
\item If 
\begin{equation}
\label{R_T}\left\{ \left(\frac{\pi}{2} - |\theta| \right)^\alpha \exp(-e^{i\theta} t A) :\: t > 0,\: \theta \in \left(-\frac{\pi}{2},\frac{\pi}{2} \right) \right\}\text{ is }R\text{-bounded}
\end{equation}
and $X$ has property $(\alpha),$ then $A$ has an $R$-bounded $\Hor^\beta_2$ calculus for $\beta > \alpha + \frac12.$
\item If 
\begin{equation}
\label{N_T}
R\left( \left\{ \exp(-e^{i\theta} t 2^n A) :\: n \in \Z \right\} \right) \leq C (\frac{\pi}{2} - |\theta|)^{-\alpha}
\end{equation}
with a constant $C < \infty$ uniformly in $t > 0$ and $\theta \in (-\frac{\pi}{2},\frac{\pi}{2}),$
then $A$ has a bounded $\Hor^\beta_r$ calculus for $\beta > \alpha + \frac1r.$
If $X$ has in addition property $(\alpha),$ then this calculus is $R$-bounded.
\end{enumerate}
Conversely, if $A$ has an $R$-bounded $\Hor^\beta_p$ functional calculus for some $p \in [1,\infty)$ and $\beta > \frac1p,$ then $\{ (\frac{\pi}{2} - |\theta|)^\beta \exp(-e^{i\theta} t A) :\: t > 0,\: \theta \in (-\frac{\pi}{2},\frac{\pi}{2}) \}$ is $R$-bounded.
\end{thm}

\begin{proof}
For part (1), we can assume w.l.o.g. that $\sigma < \frac{\pi}{2}.$
Indeed, the proof of (2) below which has weaker assumptions than (1) shows that $A$ has a bounded $\Hor^\beta_r$ calculus for $\beta > \alpha + \frac1r$ and thus a bounded $\HI(\Sigma_\sigma)$ calculus to any angle $\sigma > 0.$
We show that the assumptions imply
\begin{align}
&\text{The function }t \mapsto A^{1/2}T(e^{i\theta}t)x\text{ belongs to }\gamma(\R_+,dt,X) \tag*{$(S_T)_{\alpha+\frac12}$} \\
&\text{ with norm }\lesssim (\frac{\pi}{2} - |\theta|)^{-\alpha-\frac12} \|x\|\text{ for }x \in X\text{ and }\theta \in (-\frac{\pi}{2},\frac{\pi}{2}).\nonumber
\end{align}
Here, $\gamma(\R_+,dt,X)$ is the Gaussian space as defined for example in \cite{KaW2},  \cite[Definition 3.7]{vN}.
Then $(S_T)_{\alpha+\frac12}$ will imply by \cite[Proposition 3.3]{LM2} and \cite[Corollary 3.19]{HaKu} that
\[ \left\{ \int_0^\infty f(t) A^{\frac12} T(e^{i\theta}t) dt :\: \|f\|_{L^2(\R_+,dt)} \leq 1 \right\} \text{ is }R\text{-bounded with }R\text{-bound }\lesssim (\frac{\pi}{2} - |\theta|)^{-\alpha - \frac12}\]
for $|\theta| < \frac{\pi}{2},$ which implies by \cite{KrW2} that $A$ has an $R$-bounded $\Hor^{\beta}_2$ calculus for any $\beta > \alpha + \frac12.$

By \cite[Theorem 7.2, Proposition 7.7]{KaW2}, the fact that $A$ has an $\HI(\Sigma_\sigma)$ calculus with angle $\sigma < \frac{\pi}{2}$ implies that for $x \in D(A) \cap R(A),$
$\|A^{\frac12}T(t)x\|_{\gamma(\R_+,X)} \lesssim \|x\|,$
so that by \cite[Corollary 6.3]{vN} with the isomorphic mapping $L^2(\R_+,dt) \to L^2(\R_+,dt), t \mapsto ts$
\[ \| A^{\frac12} T(ts) x \|_{\gamma(\R_+,X)} \lesssim s^{-\frac12} \|x\| \]
for $s > 0.$
Decompose
\[ e^{\pm i(\frac{\pi}{2} - \omega)} = r e^{\pm i (\frac{\pi}{2}-\frac{\omega}{2})} + s\]
where $s,r > 0$ are uniquely determined.
By the law of sines, $s \cong \omega$ for $\omega \to 0+.$
Then
\[A^{\frac12} T(te^{\pm i (\frac{\pi}{2} -\omega)}) = T(tre^{\pm i (\frac{\pi}{2}-\frac{\omega}{2})}) \circ A^{\frac12} T(ts).\]
Therefore, by assumption of part (1),
\begin{align*}
\| A^{\frac12} T(te^{\pm i (\frac{\pi}{2} - \omega)}) x \|_{\gamma(\R_+,X)}
& \leq R(\{T(tre^{\pm i (\frac{\pi}{2} - \frac{\omega}{2})}):\:t > 0\}) \| A^{\frac12} T(ts) x\|_{\gamma(\R_+,X)}
\\ & \lesssim \omega^{-\alpha} \| A^{\frac12} T(ts)x\|_{\gamma(\R_+,X)}
\\ & \lesssim \omega^{-\alpha} \omega^{-\frac12} \|A^{\frac12}T(t)x\|_{\gamma(\R_+,X)}
\\ & \lesssim \omega^{-\alpha - \frac12} \|x\|.
\end{align*}
Now $(S_T)_{\alpha + \frac12}$ follows and part (1) is proved.

We now turn to part (2).
Note that \eqref{N_T} implies that $A$ is $R$-sectorial \cite[2.16 Example and 2.20 Theorem]{KuWe}.
Then similarly to the proof of Theorem \ref{Thm Sufficient conditions BIP}, (2) follows from Theorem \ref{Thm Localization Principle}, if we can show that \eqref{Equ Localized R-bound} holds with $\alpha$ replaced by $\beta$ and $p$ replaced by $r.$
To this end, let $f \in C^\infty_c(\R_+)$ with $\supp f \subset [\frac12,2]$ and $\|f\|_{\Sobolev^\beta_r} \leq 1.$
Write $\delta = \beta - \alpha > \frac{1}{r}.$
The assumptions imply by an inspection of the proof of \cite[End of 2.16 Example]{KuWe} that the $R$-boundedness assumption of (1) hold for some larger $\alpha'.$
This in turn implies by (a simplified version of) the proof of part (1) (see also Remark \ref{rem-semigroup-auxiliary-calculus}) that $A$ has a $\Sobexp^{\alpha'}_p$ calculus.
Then by Lemma \ref{Lem Sobolev calculus} (3), for $x \in D_A$ and $n \in \Z,$
\[ f(2^n A)x = \frac{1}{2\pi} \int_{\R} \check{f}(s) \exp(-is2^nA)x ds = \frac{1}{2\pi} \int_{\R} \check{f}(s-i) \exp((-is-1)2^nA)x ds. \]
Here we performed a shift of a complex contour integral, which is allowed, since $|\check{f}(s-it)|$ decays faster than any polynomial for $|s|\to \infty$ and $t \in [0,1],$ and $\| \exp((-is-1)2^nA)x \|$ grows only polynomially as $|s|\to\infty.$
We decompose the integrand as
\begin{equation}\label{Equ Proof Prop norm Cr} \left[ \frac{1}{|-is-1|^{\alpha+\delta}} \exp((-is-1)2^nA)x \right] \left[ \check{f}(s-i) |-is-1|^{\alpha+\delta} \right].
\end{equation}
Consider the first bracket as a function in the variable $s \in \R,$ with values in $B(\Rad(X)),$ where $\Rad(X)$ is the closed subspace of $L^2(\Omega,X)$ which is generated by elements of the form $\epsilon_n \otimes x$ with $x \in X$ and $(\epsilon_n)_{n \in \Z}$ a sequence of independent Rademacher variables over the probability space $\Omega$ (cf. Section \ref{Sec Prelims Rad}).
This means that $\sum_{n \in \Z} \epsilon_n \otimes x_n \mapsto \sum_{n \in \Z} \epsilon_n \otimes \frac{1}{|-is-1|^{\alpha+\delta}} \exp((-is-1)2^nA)x_n.$
As $|-is-1| \cong (\frac{\pi}{2} - |\theta|)^{-1}$ for the choice $s = \tan \theta$ and $t^2 = 1 + s^2,$ so that $e^{i\theta}t = 1 + is,$ one sees that the assumption in (2) implies that $\{ \frac{1}{|-is-1|^\alpha} \exp((-is-1)2^n A) :\: n \in \Z \}$ is $R$-bounded over $X$ with uniform $R$-bound in $s \in \R.$ 
Thus, 
\[ \left\| \frac{1}{|-is-1|^\delta} \left\| \left(\frac{1}{|-is-1|^\alpha} \exp((-is-1)2^n A)\right)_{n \in \Z} \right\|_{B(\Rad(X))} \right\|_{L^r(\R,ds)}  \lesssim \| \frac{1}{|-is-1|^\delta} \|_{L^r(\R,ds)}, \]
which is finite by the choice $\delta r > 1.$
The $L^{r'}(\R,ds)$ norm of the expression in the second bracket of \eqref{Equ Proof Prop norm Cr} is estimated by
\[
\| \check{f}(s-i) |-is-1|^{\alpha + \delta} \|_{L^{r'}(\R,ds)} = \| [f \exp]\ck(s) |-is-1|^{\alpha + \delta} \|_{L^{r'}} \lesssim \|f \exp\|_{\Sobolev^\beta_r} \lesssim \|f\|_{\Sobolev^\beta_r},
\]
where the last estimate follows from $\supp f \subset [\frac12,2].$
Note that $\Rad(X)$ has the same type and cotype as $X$ as a closed subspace of $L^2(\Omega,X).$
By Proposition \ref{Prop Hytonen Veraar}, it follows that
\[ \left\{ \left( f(2^n A) \right)_{n \in \Z} :\: f \in C^\infty_c, \: \supp f \subset [\frac12, 2],\: \|f\|_{\Sobolev^\beta_r} \leq 1 \right\}\text{ is }R\text{-bounded over }\Rad(X). \]
This implies that also
\begin{equation}\label{Equ Proof Thm Sufficient conditions Sgr}
\left\{ f(2^n A) :\: f \in C^\infty_c, \: \supp f \subset [\frac12, 2],\: \|f\|_{\Sobolev^\beta_r} \leq 1, \: n \in \Z \right\}\text{ is }R\text{-bounded over }X.
\end{equation}
Indeed, let $x_i \in X,\: n_i \in \Z$ and $f_i \in C^\infty_c$ such that $\supp f_i \subset [\frac12,2]$ and $\|f_i\|_{\Sobolev^\beta_r} \leq 1.$
Put $y_i = \epsilon_{n_i} \otimes x_i \in \Rad(X).$ 
Then, with $(\epsilon'_i)_{i}$ being another sequence of independent Rademachers over a different probability space $\Omega',$ we have
\begin{align*}
\E' \| \sum_i \epsilon_i' f_i(2^{n_i}A) x_i \| & = \E \E' \| \sum_i \epsilon_i' \epsilon_{n_i} f_i(2^{n_i}A) x_i \| = \E \E' \| \sum_i \epsilon_i' (f_i(2^nA))_{n \in \Z} (y_i) \| \\
& \lesssim \E \E' \| \sum_i \epsilon_i' y_i \| = \E' \| \sum_i \epsilon_i' x_i \|.
\end{align*}
This shows \eqref{Equ Proof Thm Sufficient conditions Sgr}, and thus \eqref{Equ Localized R-bound}.\\

The converse statement follows by applying Lemma \ref{Lem Mihlin norms of functions}.
\end{proof}

\begin{rem}
\label{rem-semigroup-auxiliary-calculus}
In Theorem \ref{Thm Sufficient conditions Sgr}, the H\"ormander functional calculus is a consequence of $R$-boundedness conditions on the analytic semigroup generated by $-A.$
If one weakens the assumptions to norm bound conditions as in the following, one gets by the same technique a weaker Sobolev functional calculus, but not necessarily a $\Hor^\alpha_p$ calculus, cf. Subsection \ref{subsec-Kalton}.

Let $A$ be a $0$-sectorial operator on some Banach space $X.$
Assume that for some $\alpha > 0,$
\[ \|\exp(-zA)\| \leq C \left( \frac{|z|}{\Re z} \right)^\alpha \quad (z \in \C_+). \]
Then for each $R > 0,$ the set $\{ f(A) : \: f \in \Sobolev^\beta_r,\: \supp f \subset [0,R],\:\|f\|_{\Sobolev^\beta_r} \leq 1 \}$ is $R$-bounded,
where $r \leq 2,\: \frac1r > \frac{1}{\type X} - \frac{1}{\cotype X}$ and $\beta > \alpha + \frac1r.$
Furthermore, $A$ has an auxiliary calculus $\Phi_A : \Hor^\beta_r \to B(D(\theta),X)$ for any $\theta > 0.$
\end{rem}

\begin{proof}
Let $\delta > \frac1r > \frac{1}{\type X} - \frac{1}{\cotype X}$ and $\beta = \alpha + \delta.$
Clearly the assumptions imply that $\{ |1 + is|^{-\alpha} \exp(-(1+is)A):\: s \in \R \}$ is bounded.
Then as in the proof of Theorem \ref{Thm Sufficient conditions Sgr}, we have for $f \in \Sobolev^{\beta}_r$
\begin{align*}
f(A)x & = \frac{1}{2\pi} \int_\R \check{f}(s-i) \exp(-(is+1)A)x ds \\
& = \frac{1}{2\pi} \int_\R \left[\check{f}(s-i) \langle s \rangle^{\beta}\right] \left[ \langle s \rangle^{-\delta}\langle s \rangle^{-\alpha} \exp(-(is+1)A)x\right] ds.
\end{align*}
The second bracket is in $L^r$ by the assumption $\delta r > 1.$
The first bracket is in $L^{r'}(\R)$ with
$\|\check{f}(s-i) \langle s \rangle^{\alpha + \delta}\|_{L^{r'}} \leq \|f \exp\|_{\Sobolev^\beta_r} \lesssim \|f\|_{\Sobolev^\beta_r}$ by the Hausdorff-Young inequality, as soon as $f$ has support in $[0,R].$
Then the first statement follows with Proposition \ref{Prop Hytonen Veraar}.

For the second statement, note that the first part shows $\|f(A)\| \lesssim \|f\|_{\Sobolev^\beta_r} \cong \|f\|_{\Sobexp^\beta_r}$ for any $f$ with $\supp f \subset [\frac12,2].$
Since $2^{-n} A$ satisfies the same assumptions as $A$ for any $n \in \Z,$ we deduce
$\|f(2^{-n}A)\| \lesssim \|f\|_{\Sobexp^\beta_r}$ for any $f$ with $\supp f \subset [\frac12,2],$ uniformly in $n \in \Z.$
Now we take some dyadic partition of unity $(\dyad_n)_{n \in \Z},$ $\theta > 0$ and recall $\rho(\lambda) = \lambda / (1 + \lambda)^2.$
We have
\begin{align*}
\| (\rho^\theta f)(A) \| & = \| \sum_{n \in \Z} (\dyad_n \rho^\theta f)(A) \| = \| \sum_{n \in \Z} [\dyad_0 \rho^\theta(2^n \cdot) f(2^n \cdot) ](2^{-n} A) \| \\
& \lesssim \sum_{n \in \Z} \| \dyad_0 \rho^\theta(2^n \cdot) f(2^n \cdot) \|_{\Sobexp^\beta_r} \lesssim \sum_{n \in \Z}\|\dyad_0 \rho^{\theta/2}(2^n \cdot)\|_{\Sobexp^\beta_r} \|\rho^{\theta/2}(2^n \cdot) f(2^n \cdot)\|_{\Sobexp^\beta_r} \\
& \lesssim \sum_{n \in \Z} 2^{-|n|\theta/2} \|\rho^{\theta/2}f\|_{\Sobexp^\beta_r} \lesssim \|f\|_{\Hor^\beta_r},
\end{align*}
where the last but one estimate follows similarly to the proof of Theorem \ref{Thm Localization Principle}, and the last estimate follows from the argument in the proof of Lemma \ref{Lem auxiliary calculus multiplicative}.
Thus, the auxiliary calculus $\Phi_A : \Hor^\beta_r \to B(D(\theta),X)$ is bounded.
\end{proof}

\begin{rem}\label{Rem Norm cond bad Mihlin}~
The connection of the exponent of the $\Hor^\alpha_\infty$ calculus with the growth rate of the analytic semigroup when approaching the imaginary axis as in condition (1) or (2) of Theorem \ref{Thm Sufficient conditions Sgr} has been studied already by Duong in \cite[Section 3]{Duon}.
There, the calculus of a Laplacian operator on a Nilpotent Lie group is investigated.
A H\"ormander calculus is obtained from a kernel estimate of the analytic semigroup.
Note that in Duong's situation, Gaussian estimates for the semigroup are at hand.
For a comparison of Theorem \ref{Thm Sufficient conditions Sgr} with spectral multiplier theorems for Gaussian estimates, we refer to Subsection \ref{Subsec GGE}.
\end{rem}

Next we compare conditions on the wave operator associated with $A,$ i.e. (variants of) the boundary value on the imaginary axis of the analytic semigroup generated by $-A,$ with conditions on the analytic semigroup.
Consider the following assertions.
\begin{align}
\left\{( 1 + |s|2^nA)^{-\alpha} e^{i s 2^nA} : \: n \in \Z \right\}&\text{ is $R$-bounded uniformly in $s \in \R$.}
\label{N_W} \\
\left\{ ( 1 + |s|A)^{-\alpha} e^{isA} : \: s \in \R \right\} & \text{ is }R\text{-bounded.} \label{R_W}
\end{align}
Note that these assertions include that the operators in question are defined on $X$ and bounded.
They are well-defined operators at least on the domain $D_0 = R(e^{-A}),$ (which is dense in $X$ by the analyticity of the dual semigroup $(e^{-tA})'),$ by the formula $(1 + |s| A)^{-\alpha} e^{isA}x = (1 + |s|A)^{-\alpha} e^{(is-1)A}y$ for $x = e^{-A}y \in D_0.$
Operators as in \ref{R_W} are considered in \cite{dLL}, where they are called regularized semigroup, and in particular in \cite[Sections 7.3, 7.4.2]{Ouha}, \cite[Theorems 2,3]{Mu} in connection with spectral multipliers.
Moreover, the link with analytic semigroups on the right half plane is studied in \cite[Theorems 2.2, 2.3]{Bo}.
Clearly, \eqref{R_W} implies \eqref{N_W}.

We have the following sufficient conditions for the $\Hor^\beta_r$ calculus in terms of the wave operators.

\begin{thm}\label{Thm Sufficient conditions Wave}
Let $A$ be a $0$-sectorial operator on a Banach space $X$ with property $(\alpha)$ having an $\HI(\Sigma_\sigma)$ calculus for some $\sigma \in (0,\frac{\pi}{2}).$
Let $r \in (1,2],\: \frac1r > \frac{1}{\type X} - \frac{1}{\cotype X}.$
(See Remark \ref{rem-type-Lp} for the $L^q$-case.)
\begin{enumerate}
\item If $\{ ( 1 + |s| A)^{-\alpha} \exp(isA) :\: s \in \R \}$ is $R$-bounded, then $A$ has an $R$-bounded $\Hor^\beta_2$ calculus for $\beta > \alpha + \frac12.$
\item If $\{ (1 + |s| 2^n A)^{-\alpha} \exp(is 2^n A) :\: n \in \Z \}$ is $R$-bounded uniformly in $s \in \R,$ then $A$ has an $R$-bounded $\Hor^\beta_r$ calculus for $\beta > \alpha + \frac1r.$
\end{enumerate}
Conversely, if $A$ has an $R$-bounded $\Hor^\beta_p$ functional calculus for some $\beta$ and $p \in [1,\infty)$ such that $\beta > \frac1p,$ then $\{ (1 + |s| A)^{-\beta} \exp(isA) :\: s \in \R \}$ is $R$-bounded.
\end{thm}

\begin{proof}
By Proposition \ref{Prop N_T N_W almost equivalent} below, the assumptions in (1) (resp. (2)) imply the assumptions in (1) (resp. (2)) of Theorem \ref{Thm Sufficient conditions Sgr}, so that (1) and (2) above follow immediately.
For the converse statement, we refer again to Lemma \ref{Lem Mihlin norms of functions}.
\end{proof}

\begin{prop}\label{Prop N_T N_W almost equivalent}
Let $A$ have
a bounded $\HI(\Sigma_\sigma)$ calculus for some $\sigma \in (0,\frac{\pi}{2}).$
Let the underlying Banach space have property $(\alpha).$
Then for $\alpha > 0,$ we have
\[ \eqref{R_W} \Longrightarrow \eqref{R_T}\text{ and }\eqref{N_W} \Longrightarrow \eqref{N_T}.\]
\end{prop}

\begin{proof}
$\eqref{R_W} \Longrightarrow \eqref{R_T}$:
Note first that for any $\omega \in (\sigma,\frac{\pi}{2}),$ we have
\begin{equation}\label{Equ Aux 4}
\{f(A):\:\|f\|_{\infty,\omega} \leq 1\}\text{ is }R\text{-bounded.}
\end{equation}
Indeed, since $X$ has property $(\alpha),$ by \cite[Theorem 12.8]{KuWe}, \eqref{Equ Aux 4} follows.
In particular,
\[
\left\{(\frac{\pi}{2} - | \theta|)^\alpha T(e^{i\theta}t):\:t>0,\,|\theta| \leq \frac{\pi}{2} - \omega \right\}\text{ is }R\text{-bounded.}
\]
Fix some $\omega \in (\sigma,\frac{\pi}{2})$ for the rest of the proof.
Thus it remains to show that
\begin{equation}\label{Equ 1a Proof Lem N_T}
\left\{ (\frac{\pi}{2} - | \theta|)^\alpha T(e^{i\theta} t):\:t>0,|\theta| \in (\frac{\pi}{2}-\omega,\frac{\pi}{2})\right\}\text{ is }R\text{-bounded.}
\end{equation}
We write $e^{i\theta} t = r + is$ with real $r$ and $s.$
Then for $x = e^{-A}y \in D_0,$ we have
\begin{align*}
&\left(\frac{r}{|s|}\right)^\alpha T(r+is) x = \left(\frac{r}{|s|}\right)^\alpha T(r+1+is) y \\
& = \left[ (1+|s| A)^{-\alpha} e^{(-is-1) A}\right] \circ \left[\left(\frac{r}{|s|}\right)^\alpha (1+r A)^{-\alpha} (1+|s| A)^\alpha \right] \circ \left[ (1+ rA)^\alpha T(r)\right]y \\
& = \left[ (1+|s| A)^{-\alpha} e^{-is A}\right] \circ \left[\left(\frac{r}{|s|}\right)^\alpha (1+r A)^{-\alpha} (1+|s| A)^\alpha \right] \circ \left[ (1+ rA)^\alpha T(r)\right]x .
\end{align*}
We show that all three brackets form $R$-bounded sets for $r+is$ varying in $\{ z \in \C\backslash \{ 0 \}:\: |\arg z| \in (\frac{\pi}{2}-\omega,\frac{\pi}{2}) \}.$
Note that for $|\theta| \in (\frac{\pi}{2}-\omega,\frac{\pi}{2}),$ we have $\frac{\pi}{2} - |\theta| \cong \frac{r}{|s|},$
so that this will imply \eqref{Equ 1a Proof Lem N_T} by Kahane's contraction principle.
The assumption \eqref{R_W} implies that the first bracket is $R$-bounded with $s$ varying in $\R.$
We show in a moment that
\begin{equation}\label{Equ 2a Proof Lem N_T}
\left(\frac{r}{|s|}\right)^\alpha(1+r (\cdot))^{-\alpha} (1+|s|(\cdot))^\alpha\text{ is uniformly bounded in }\HI(\Sigma_{\frac{\pi}{2}}).
\end{equation}
Then the fact that the second bracket is $R$-bounded follows from \eqref{Equ Aux 4}.
For $\lambda \in \Sigma_{\frac{\pi}{2}},$
\begin{align*}
 \left(\frac{r}{|s|}\right)^\alpha \left|(1+r \lambda)^{-\alpha} (1+|s|\lambda)^\alpha\right|
& = \left| \frac{\frac{1}{|s|}+\lambda}{\frac1r + \lambda} \right|^\alpha
\\ & \cong \left(\frac{\frac{1}{|s|}+|\lambda|}{\frac1r + |\lambda|}\right)^\alpha
\\ & \lesssim 1,
\end{align*}
since $|s| \gtrsim r$ by the restriction $|\theta| \in (\frac{\pi}{2}-\omega,\frac{\pi}{2}).$
Thus, \eqref{Equ 2a Proof Lem N_T} follows.
Finally, \eqref{Equ Aux 4} with $f(\lambda) = (1+\lambda)^\alpha e^{-\lambda}$ implies that the third bracket is $R$-bounded with $r$ varying in $(0,\infty).$
Now \eqref{R_T} follows since $x$ is from the dense subspace $D_0.$\\

\noindent
$\eqref{N_W} \Longrightarrow \eqref{N_T}:$
The proof is similar to $\eqref{R_W} \Longrightarrow \eqref{R_T}.$
\end{proof}

\begin{rem}
If $A$ satisfies all the assumptions of Theorem \ref{Thm Sufficient conditions Sgr} or \ref{Thm Sufficient conditions Wave}, parts (1) or (2), but is not injective, then one has the following variants for a reflexive Banach space $X.$
The semigroup $\exp(-zA)$ leaves $N(A)$ and $\overline{R(A)}$ invariant and $\exp(-zA)|_{\overline{R(A)}}$ is again an analytic semigroup with generator $A_1,$ where $A = A_1 \oplus 0$ is the decomposition on $\overline{R(A)} \oplus N(A)$ from Subsection \ref{Subsec A B}.
Moreover, this semigroup satisfies the conditions of (1) or (2) on the space $\overline{R(A)}$ in Theorem \ref{Thm Sufficient conditions Sgr}.
Thus also the conclusions of (1) and (2) hold for $A_1$ and $A.$
A similar remark holds for (1) and (2) of Theorem \ref{Thm Sufficient conditions Wave}.
\end{rem}

\begin{rem}
We end this section with a sufficient condition for the H\"ormander calculus in terms of an $R$-boundedness condition on resolvents.
It is not optimal: we loose one order in the differentiation parameter instead of $\frac12$ when passing from $R$-bounded resolvents to the functional calculus.

Let $A$ be a $0$-sectorial operator with $\HI$ calculus on a space with property $(\alpha).$
Consider the folllowing condition.
\begin{equation}
R\left(\{ tR(e^{i\theta}t,A) :\: t > 0\} \right) \lesssim |\theta|^{-\alpha_1} \quad (0 < |\theta| < \pi) \label{Equ R Resolvent}.
\end{equation}
Then \eqref{Equ R Resolvent} implies that $A$ has an $R$-bounded $\Hor^{\alpha_1}_2$ calculus.
Conversely, an $R$-bounded $\Hor^{\alpha}_2$ calculus of $A$ implies that \eqref{Equ R Resolvent} holds with $\alpha_1 = \alpha + 1.$
\end{rem}

\begin{proof}
Using the resolvent identity one can prove in a similar way as in Theorem \ref{Thm Sufficient conditions BIP} that \eqref{Equ R Resolvent} $\Longrightarrow$ \eqref{Equ Square Resolvent} with $\alpha_1 = \alpha_2,$ and
\begin{equation}
\| A^{\frac12} R(e^{i\theta}t,A) x \|_{\gamma(\R_+,dt,X)} \lesssim |\theta|^{-\alpha_2} \|x\| \quad ( 0 < |\theta| < \pi) \label{Equ Square Resolvent}.
\end{equation}
Further, \eqref{Equ Square Resolvent} implies that $A$ has an $R$-bounded $\Hor^{\alpha_2}_2$ calculus.
For this fact and the definition of $\gamma(\R_+,dt,X),$ we refer to \cite{Kr2}.
The converse statement follows from a H\"ormander norm estimate of $t(e^{i\theta}t - (\cdot))^{-1}.$
\end{proof}

\section{Examples and Counterexamples}\label{Sec Examples Counterexamples}

\subsection{Comparison with the H\"ormander theorem for the Laplace operator on $\R^n$}
\label{subsec-comparison}

It is known that the classical H\"ormander spectral multiplier theorem for the Laplace operator on $L^p(\R^d),\: 1 < p < \infty,$ can be improved as follows (see e.g. \cite[Theorem A, p.~7]{COSY}, \cite{Tomas}).
$-\Delta$ has a $\Hor^\beta_2$ calculus on $L^p(\R^d)$ with
\[ \beta > \beta_p := \max \left(\frac{d}{d+1}, d \left|\frac1p - \frac12 \right| \right).\]
How close can we come to this result with our general method?
By the well known $R$-boundedness of the Gaussian kernel (see e.g. \cite[Section 5.4]{KuWe}) we get
\begin{equation}
\label{equ-Laplace-R-bdd-sgr}
R\left(\left\{ \exp(-e^{i\theta} 2^n tA) :\: n \in \Z \right\}\right) \leq C \left( \frac{\pi}{2} - \left|\theta \right| \right)^{-\alpha} \quad \left(t > 0,\: \theta \in \left(-\frac{\pi}{2},\frac{\pi}{2} \right) \right)
\end{equation}
for $A = -\Delta$ and all $p \in (1,\infty),$ and $\alpha = \frac{d}{2}.$
By complex Stein interpolation between $p = 2$ (where $\alpha = 0$) and $p$ close to $1$ or $\infty,$ we can improve the exponent in \eqref{equ-Laplace-R-bdd-sgr} to all $\alpha$ with \[\alpha > \alpha_p = d \left| \frac1p - \frac12 \right|.\]
Then by Theorem \ref{Thm Sufficient conditions Sgr}, $-\Delta$ has a $\Hor^\beta_2$ calculus on $L^p(\R^d)$ of order $\beta > \alpha_p + \frac12.$
Hence our theorem, when applied to $-\Delta$ directly, overestimates the necessary smoothness order for $\beta_p > \frac{d}{d+1}$ by
\[ \alpha_p + \frac12 - \beta_p = \frac12. \]
This gap can be narrowed by considering $A = (-\Delta)^{\frac12}$ instead of $-\Delta.$
Since $A^2 = - \Delta,$ both operators have a H\"ormander calculus of the same order.
But the Poisson semigroup generated by $A$ has an $R$-bound (see \cite[Section 4]{Kr3}) which gives \eqref{equ-Laplace-R-bdd-sgr} with $\alpha = \frac{d-1}{2}$ for all $p \in (1,\infty).$
Again by complex Stein interpolation, we get \eqref{equ-Laplace-R-bdd-sgr} for any exponent larger than
\[ \alpha_p' = (d-1) \left|\frac1p - \frac12 \right|. \]
Theorem \ref{Thm Sufficient conditions Sgr} guarantees now a $\Hor^\alpha_2$ calculus for $A$ or order $\alpha_p' + \frac12.$
Hence the gap $\alpha_p' + \frac12 - \beta_p = \frac12 - | \frac1p - \frac12 |,$ disappears for $p \to \infty$ or $p \to 1.$
In particular, Theorem \ref{Thm Sufficient conditions Sgr} implies a $\Hor^{d/2}_2$-calculus for the Laplace operator on $L^p(\R^d)$ for all $1 < p < \infty.$
The sharp result with $\beta > \beta_p$ mentioned in the beginning requires more advanced methods, e.g. see \cite{COSY}.
For $|\frac1p - \frac12| < \frac{1}{d+1},$ the optimal order of the H\"ormander calculus still seems to be unknown.

\subsection{(Generalized) Gaussian estimates imply the $R$-boundedness of the semigroup}
\label{Subsec GGE}

The considerations of Subsection \ref{subsec-comparison} can be extended to (generalized) Gaussian estimates.
We compare them with the $R$-boundedness condition of the semigroup as in Theorem \ref{Thm Sufficient conditions Sgr}, and compare our spectral multiplier theorem with the ones in the literature.

\begin{defi}
Let $\Omega$ be a topological space which is equipped with a distance $\rho$ and a Borel measure $\mu.$
Let $d \geq 1$ be an integer.
$\Omega$ is called a homogeneous space\index{homogeneous space} of dimension $d$ if there exists $C > 0$ such that for any $x \in \Omega,\,r> 0$ and $\lambda \geq 1:$
\[\mu(B(x,\lambda r)) \leq C \lambda^d \mu(B(x,r)).\]
\end{defi}

Typical cases of homogeneous spaces are open subsets of $\R^d$ with Lipschitz boundary and Lie groups with polynomial volume growth, in particular stratified nilpotent Lie groups (see e.g. \cite{FoSt}).

We will consider operators satisfying the following assumption.

\begin{assumption}\label{Ass Examples}
$A$ is a self-adjoint positive (injective) operator on $L^2(\Omega),$ where $\Omega$ is a homogeneous space of a certain dimension $d.$
Further, there exists some $p_0 \in [1,2)$ such that the semigroup generated by $-A$ satisfies the so-called generalized Gaussian estimate
(see e.g. \cite[(GGE)]{Bluna}):
\begin{equation}\tag{GGE}\label{Equ generalized Gaussian estimate}
\| \chi_{B(x,r_t)} e^{-tA} \chi_{B(y,r_t)} \|_{p_0 \to p_0'} \leq C \mu(B(x,r_t))^{\frac{1}{p_0'}-\frac{1}{p_0}} \exp \left(-c \left(\rho(x,y)/r_t \right)^\frac{m}{m-1} \right) \quad (x,y \in \Omega,\,t>0).
\end{equation}
Here, $p_0'$ is the conjugated exponent to $p_0,\,C,c > 0,\,m \geq 2$ and $r_t = t^{\frac1m},$ $\chi_B$ denotes the characteristic function of a $B,$ where $B(x,r) = \{y \in \Omega:\: \rho(y,x) < r\}$
and $\|\chi_{B_1} T \chi_{B_2}\|_{p_0 \to p_0'} = \sup_{\|f\|_{p_0} \leq 1} \|\chi_{B_1} \cdot T(\chi_{B_2} f)\|_{p_0'}.$
\end{assumption}

If $p_0 = 1,$ then it is observed in \cite{BlKub} that \eqref{Equ generalized Gaussian estimate} is equivalent to the usual Gaussian estimate,
i.e. $e^{-tA}$ has an integral kernel $k_t(x,y)$ satisfying the pointwise estimate (cf. e.g. \cite[Assumption 2.2]{DuOS})
\begin{equation}\tag{GE}\label{Equ Gaussian estimate}
|k_t(x,y)| \lesssim \mu(B(x,t^{\frac1m}))^{-1} \exp\left(-c \left(\rho(x,y)/t^{\frac1m}\right)^{\frac{m}{m-1}}\right)\quad (x,y \in \Omega,\,t>0).
\end{equation}
This is satisfied in particular by sublaplacian operators on Lie groups of polynomial growth \cite{Varo} as considered e.g. in
\cite{MaMe,Chri, Alex, MSt,Duon}, or by more general elliptic and sub-elliptic operators
\cite{Davia, Ouha}, and Schr\"odinger operators \cite{Ouhaa}.
It is also satisfied by all the operators in \cite[Section 2]{DuOS}.

Examples of operators satisfying a generalized Gaussian estimate for $p_0 > 1$ are higher order operators with bounded coefficients and Dirichlet boundary conditions on domains of $\R^d,$ Schr\"odinger operators with singular potentials on $\R^d$ and elliptic operators on Riemannian manifolds as listed in \cite[Section 2]{Bluna} and the references therein.

\begin{thm}\label{Thm R-bounded Blunck Hormander thm}
Assumption \ref{Ass Examples} implies that for any $p \in (p_0,p_0'),$
\[ \left\{ \left( \frac{|z|}{\Re z} \right)^\alpha \exp(-zA) : \: \Re z > 0 \right\} \text{ is }R\text{-bounded on }L^p(\Omega),\]
where $\alpha = d \left| \frac{1}{p_0} - \frac12 \right|.$
Consequently, by Theorem \ref{Thm Sufficient conditions Sgr}, $A$ has an $R$-bounded $\Hor^{\beta}_2$ calculus on $L^p(\Omega)$ for any $\beta > d \left| \frac{1}{p_0} - \frac12 \right| + \frac12.$
\end{thm}

\begin{proof}
By \cite[Proposition 2.1]{BlKua}, the assumption \eqref{Equ generalized Gaussian estimate} implies that
\[ \| \chi_{B(x,r_t)} e^{-tA} \chi_{B(y,r_t)}\|_{p_0\to 2} \leq C_1 \mu(B(x,r_t))^{\frac{1}{2}-\frac{1}{p_0}} \exp(-c_1 (\rho(x,y)/r_t)^{\frac{m}{m-1}}) \quad (x,y \in \Omega,\,t>0)\]
for some $C_1,c_1>0.$
By \cite[Theorem 2.1]{Blun}, this can be extended from real $t$ to complex $z = te^{i\theta}$ with $\theta \in (-\frac{\pi}{2},\frac{\pi}{2}):$
\[ \|\chi_{B(x,r_z)} e^{-zA} \chi_{B(y,r_z)} \|_{p_0 \to 2} \leq C_2 \mu(B(x,r_z))^{\frac{1}{2}-\frac{1}{p_0}} (\cos \theta )^{-d(\frac{1}{p_0} -\frac{1}{2})} \exp(-c_2 (\rho(x,y)/r_z)^{\frac{m}{m-1}}),\]
for $r_z = (\cos \theta )^{-{\frac{m-1}{m}}} t^{\frac1m},$ and some $C_2,c_2 >0.$
By \cite[Proposition 2.1 (i) (1) $\Rightarrow$ (3) with $R = e^{-zA},\,\gamma =\alpha = \frac{1}{p_0} - \frac12,\,\beta = 0,\,r = r_z,\,u=p_0$ and $v = 2$]{BlKua}, this gives for any $x \in \Omega,\,\Re z > 0$ and $k \in \N_0$
\[ \|\chi_{B(x,r_z)} e^{-zA} \chi_{A(x,r_z,k)}\|_{p_0 \to 2} \leq C_3 \mu(B(x,r_z))^{\frac12 - \frac{1}{p_0}} (\cos \theta )^{-d(\frac{1}{p_0} -\frac{1}{2})} \exp(-c_3 k^{\frac{m}{m-1}}),\]
where $A(x,r_z,k)$ denotes the annular set $B(x,(k+1)r_z)\backslash B(x,k r_z).$
By \cite[Theorem 2.2 with $q_0 = p_0,\,q_1 = s = 2,\rho(z) = r_z$ and $S(z) = (\cos \theta)^{d(\frac{1}{p_0}-\frac{1}{2})} e^{-zA}$]{Kuns} and property $(\alpha),$
we deduce that
\[ \{ (\cos \theta)^{d(\frac{1}{p_0}-\frac{1}{2})} e^{-zA} :\: \Re z > 0\} \]
is $R$-bounded.
The first part of the theorem is shown.
Then the second part follows from Theorem \ref{Thm Sufficient conditions Sgr}, noting that $A$ has an $\HI$ calculus on $L^p(\Omega)$ \cite[Corollary 2.3]{Blun}.
\end{proof}

\begin{rem}\label{Rem R-bounded Blunck Hormander thm}
Theorem \ref{Thm R-bounded Blunck Hormander thm} improves (if $p_0 > 1$) the smoothness order of the spectral multiplier theorem in \cite[Theorem 1.1]{Bluna} with same assumptions, from $\displaystyle \frac{d}{2} + \frac{1}{2} +\epsilon$ in \cite{Bluna} to $\displaystyle d \left|\frac{1}{p_0} - \frac{1}{2} \right| + \frac12 + \epsilon.$
Note that \cite{Bluna} obtains also a weak-type result for $p = p_0.$
In \cite{KuU}, \cite[Theorem 6.4 a)]{Uhl}, under the assumptions of Theorem \ref{Thm R-bounded Blunck Hormander thm}, a $\Hor^{\gamma}_r$ calculus with $\gamma > (d+1) | \frac{1}{p} - \frac{1}{2} |$ and $r > | \frac12 - \frac1p |^{-1}$ is derived.
Note that $\Hor^{\gamma}_r$ is larger than $\Hor^\beta_2$ by the Sobolev estimate in Lemma \ref{Lem Elementary Mih Hor}.
In the classical case of Gaussian estimates, i.e. $p_0 = 1,$ \cite{DuOS} yields a $\Hor^{\alpha_2}_\infty$ calculus under Assumption \ref{Ass Examples}
and even a $\Hor^{\alpha_2}_2$ calculus for many examples, e.g. homogeneous operators, with the better derivation order $\alpha_2 > \frac{d}{2}.$
In \cite{COSY}, further improvements on the differentiation order are obtained by making additional assumptions on the operator $A,$ e.g. ``restriction estimates''.
However, Theorem \ref{Thm R-bounded Blunck Hormander thm} improves on all the above cited spectral multiplier theorems in that it includes the $R$-boundedness of the H\"ormander calculus.
\end{rem}

\begin{rem}
The theorem also holds for the weaker assumption that $\Omega$ is an open subset of a homogeneous space $\tilde{\Omega}.$
In that case, the ball $B(x,r_t)$ on the right hand side in \eqref{Equ generalized Gaussian estimate} is the one in $\tilde{\Omega}.$
This variant can be applied to elliptic operators on irregular domains $\Omega \subset \R^d$ as discussed in \cite[Section 2]{Bluna}.
\end{rem}

\begin{rem}
\label{rem-selfadjoint-interpolation}
In Theorem \ref{Thm R-bounded Blunck Hormander thm}, the operator $A$ was assumed to be self-adjoint, and thus, admits a functional calculus $L^\infty \to B(L^2(\Omega)).$
The space $L^\infty = L^\infty((0,\infty);d\mu_A)$ is larger than $\Ha,$ and one can use this fact to ameliorate the functional calculus of $A$ on $L^q(\Omega)$ by complex interpolation.
One obtains that under the hypotheses of Theorem \ref{Thm R-bounded Blunck Hormander thm}, $A$ has a $\Hor^\alpha_q$ calculus on $L^p(\Omega)$ with $ p \in (p_0,p_0')$ and 
\[ \alpha > d \left( \frac{1}{p_0} - \frac12 \right) \frac{ \left|\frac1p - \frac12 \right|}{\frac{1}{p_0} -\frac12 } + \frac1q \text{ and }\frac1q > \frac12 \frac{\left|\frac1p - \frac12 \right|}{\frac{1}{p_0} - \frac12} .\]
\end{rem}

\begin{rem}
In a forthcoming paper \cite{DeKr}, we will also show a H\"ormander theorem for $A \otimes \Id_Y$ on $X = L^p(\Omega;Y)$ for many self-adjoint $A$ such that $\exp(-tA)$ has Gaussian estimates, where $Y$ is any UMD Banach lattice.
\end{rem}

In the rest of this section, we show by way of examples and counterexamples to what extent our main theorems from Sections \ref{Sec 6 Hormander} and \ref{sec-Semigroup} are optimal.

\subsection{In \eqref{R_T},\eqref{N_T},\eqref{N_W},\eqref{R_W}, $R$-bounds cannot be replaced by simple norm bounds in Theorems \ref{Thm Sufficient conditions Sgr} and \ref{Thm Sufficient conditions Wave}}
\label{subsec-Kalton}

We will show this by way of counterexamples.

\begin{thm}\label{Thm Kalton counterexample 1}
Let $\alpha \in (0,1).$
Then there exists a $0$-sectorial operator on a super-reflexive space $X$ with property $(\alpha)$ such that
\begin{enumerate}
\item $A$ has a bounded $\HI(\Sigma_\omega)$ calculus for any $\omega > 0.$
\item $A$ does not have a bounded $\Hor^\beta_2$ calculus for any $\beta > \frac12.$
\item $\left\{ \left( \frac{\Re z}{z} \right)^\alpha \exp(-zA) :\: \Re z > 0\right\}$ is bounded. \item $\left\{ \left(1+|s|A \right)^{-(\alpha + \epsilon)} e^{isA} :\: s \in \R \right\}$ is bounded for any $\epsilon > 0.$  
\item The sets in (3) and (4) are not $R$-bounded even if $\alpha$ is replaced by any large number $\alpha'.$
\end{enumerate}

Moreover, there is the following variant.
Let $\theta \in (0,\pi).$
There exists a $0$-sectorial operator on a super-reflexive space $X$ with property $(\alpha)$ such that 
\begin{enumerate}
\item $A$ has a bounded $\HI(\Sigma_\omega)$ calculus for any $\omega > \alpha \theta$ but for no $\omega < \alpha \theta.$
\item Statements (2) - (5) above hold.
\end{enumerate}
\end{thm}

\begin{proof}
We first prove the variant.
We take the operator $A$ from \cite[Theorem 2.4 and Proposition 2.5]{Kal}.
It is the multiplication operator $Af(x) = e^x f(x)$ defined initially on $L^2(\R).$ 
It is then first extended to the space $\mathcal{H}_\theta$ defined as the completion of $L^2(\R)$ with respect to the norm $\|f\|_\theta^2 = \int_\R e^{-2\theta |\xi|} |\hat{f}(\xi)|^2 d\xi,$ and second to the space $X_\theta$ defined as the completion of $L^2(\R)$ with respect to the norm $\|f\|_{X_\theta} = \sup_{a \in \R} \|f \chi_{(-\infty,a]}\|_\theta.$
Finally, $A$ is regarded on the complex interpolation space $X = [L^2,X_\theta]_\alpha.$
By \cite{CwR}, $X$ is uniformly convex and thus super-reflexive.
Further, at the end of this proof we will show that $X$ has property $(\alpha).$
It is shown in \cite[p.~98]{Kal} that $A$ is $0$-sectorial with dense domain and dense range on $X.$
Further, it is shown in \cite[Proposition 2.5]{Kal} that condition (1) of the proposition holds ($A$ acting on $X$).

Since $\HI(\Sigma_\omega) \hookrightarrow \Hor^\beta_2$ for any $\omega \in (0,\pi)$ and $\beta > \frac12,$ (1) implies that $A$ cannot have a $\Hor^\beta_2$ calculus, so (2) is shown.

This also implies (5) since the $R$-boundedness of one of the sets in (3) and (4) would imply a $\Hor^\beta_2$ calculus according to Theorems \ref{Thm Sufficient conditions Sgr} and \ref{Thm Sufficient conditions Wave}.

It remains to check (3) and (4).
For (3), we remark that $\exp(-e^{i\sigma}r A)f(x) = e_{e^{i\sigma}r}(x)f(x)$ for $f \in L^2(\R)$ and $e_{e^{i\sigma}r}(x) = \exp(-e^{i\sigma}r e^x).$
Similarly to \cite[Proof of Theorem 2.4]{Kal}, we have
\[ e_{e^{i\sigma}r}f = \int_\R -e_{e^{i\sigma}r}'(x) f\chi_{(-\infty,x)}dx \]
as a Bochner integral in $L^2(\R),$ and consequently,
\[\|e_{e^{i\sigma}r}f\|_{X_\theta} \leq \int_\R |e_{e^{i\sigma}r}'(x)| dx \|f\|_{X_\theta}.\]
We have
\begin{align*}
\int_\R |e_{e^{i\sigma}r}'(x)| dx & = \int_\R re^x | \exp(-re^{i\sigma}e^x)|dx \\
& = \int_0^\infty r \exp(-r\cos(\sigma)u)du \\
& = \int_0^\infty \exp(-u) \frac{du}{\cos(\sigma)} \\
& = \frac{1}{\cos(\sigma)}.
\end{align*}
It follows that on $X_\theta,$ we have $\|\exp(-zA)\|_{X_\theta \to X_\theta} \leq \frac{|z|}{\Re z}.$
Moreover, on $L^2(\R),$ we have $\|\exp(-re^{i\sigma}A)\|_{L^2 \to L^2} = \sup_{x \in \R} |\exp(-re^{i\sigma}e^x)| = 1.$
By complex interpolation, it follows that
\[ \|\exp(-zA)\|_{X\to X} \leq \|\exp(-zA)\|_{L^2\to L^2}^{1-\alpha} \|\exp(-zA)\|_{X_\theta \to X_\theta}^\alpha \leq \left( \frac{|z|}{\Re z} \right)^\alpha.\]
Thus, (3) is shown.

For (4), we argue similarly and replace $e_{e^{i\sigma}r}$ by $f_{s,\beta}(x) = (1 + |s|e^x)^{-\beta} \exp(ise^x).$
Then 
\begin{align*}
\|f_{s,\beta}(A)\|_{X_\theta \to X_\theta} & \lesssim \int_\R |f_{s,\beta}'(x)| dx = \int_0^\infty |\frac{d}{du}[f_{s,\beta}(\log(u))]| du \\
& \lesssim \int_0^\infty (1 + |s| u)^{-(\beta + 1)} |s| + |s|(1+|s|u)^{-\beta} du \\
& = \int_0^\infty ( 1 + u )^{-(\beta + 1)} + (1 + u)^{-\beta} du \\
& < \infty
\end{align*}
as soon as $\beta > 1.$
To get an estimate on $X = [L^2,X_\theta]_\alpha,$ apply complex interpolation to the analytic family of operators $\beta \mapsto f_{s,\beta}(A),$ to deduce
\[ \|f_{s,\beta}(A)\|_{X \to X} \leq \|f_{s,0}(A)\|_{L^2 \to L^2}^{1 - \alpha} \|f_{s,1+\epsilon}(A)\|_{X_\theta \to X_\theta}^\alpha < \infty \]
for $\beta = \alpha(1+\epsilon),$ i.e. $\beta > \alpha.$

Now we prove the first part of Theorem \ref{Thm Kalton counterexample 1}, i.e. the $\HI$ calculus angle will be arbitrarily small.
Indeed, we can modify Kalton's example from \cite{Kal} in the following way.
Let $\mathcal{H}_w$ be the completion of $L^2(\R)$ with respect to the norm $\|f\|_w^2 = \int_\R w(\xi)^2 |\hat{f}(\xi)|^2 d\xi,$ where $w:  \R \to (0,1]$ is a weight function.
Hence the original example in \cite{Kal} considers $\mathcal{H}_w = \mathcal{H}_\theta$ with $w(\xi) = \exp(-\theta |\xi|).$
Further take $X_w$ the completion of $L^2(\R)$ with respect to the norm $\|f\|_{X_w} = \sup_{a \in \R} \|f \chi_{(-\infty,a]}\|_w,$
and finally $X = [L^2,X_w]_\alpha,$ the complex interpolation space.
Again, by \cite{CwR}, $X$ is uniformly convex, thus super-reflexive, and we will show at the end that $X$ has property $(\alpha).$
The operator $A$ is again the multiplication operator $Af(x) = e^x f(x).$
Then as in the first part of the proof, one can show that (3) and (4) hold.

Note that 
\[\|A^{it}\|_{X \to X} \cong \|A^{it}\|_{\mathcal{H}_{w^\alpha} \to \mathcal{H}_{w^\alpha}} \cong \sup_{\xi \in \R} \frac{w(\xi + t)^\alpha}{w(\xi)^\alpha},\]
where the first equivalence can be shown as in \cite[Proof of Proposition 2.5]{Kal} (the new weight $w^\alpha$ is the interpolated weight of $w_0(\xi) = 1$ for $L^2$ and $w_1(\xi) = w(\xi)$), and the second follows easily from the equality $(A^{it}f)\hat{\phantom{i}}(\xi) = \hat{f}(\xi - t).$

Choose now the weight $w$ such that $\sup_{\xi \in \R} \frac{w(\xi + t)^\alpha}{w(\xi)^\alpha}$ grows subexponentially in $t,$ but superpolynomially in $t,$ e.g. $w(\xi) = \exp(-\sqrt{|\xi|})$ works.
Then $A$ has a bounded $\HI$ calculus on $X$ since it has one on $L^2$ and $\mathcal{H}_w,$ and the arguments in \cite[Proofs of Theorem 2.4 and Proposition 2.5]{Kal} apply literally.
Then on the one hand the subexponential growth $\|A^{it}\|_{X \to X} \leq C_\epsilon e^{|t|\epsilon}$ implies that $A$ has an $\HI(\Sigma_\omega)$ calculus for any angle $\omega > 0,$ and the superpolynomial growth $\|A^{it}\|_{X \to X} \geq c_\beta (1 + |t|)^\beta$ for any $\beta > 0$ implies by Lemmas \ref{Lem Mihlin norms of functions} and \ref{Lem Elementary Mih Hor} that $A$ cannot have a $\Hor^\gamma_2$ calculus for any $\gamma > 0.$

Let us now show that $X = [L^2,X_w]_\alpha$ has property $(\alpha)$ for $w(\xi)$ either $e^{- \theta|\xi|}$ or $\exp(-\sqrt{|\xi|})$ (or any other bounded weight), thus finishing the proof of both parts above.
Since $X$ is super-reflexive, it has finite cotype.
Thus, according to \cite{Pis}, it suffices to show that $X$ is a subspace of a space with local unconditional structure.
We will show that
\begin{equation}
\label{equ-local-unconditional-structure}
X = [L^2,X_w]_\alpha \hookrightarrow L^\infty(\R,[L^2,\mathcal{H}_w]_\alpha),
\end{equation}
more precisely, $X$ is isomorphic to a subspace of the right hand side.
Considering functions in $L^2$ and $\mathcal{H}_w$ in their Fourier image, these spaces are weighted $L^2$-spaces, hence their complex interpolation is a weighted $L^2$ space again \cite[5.5.3 Theorem]{BeL}, and in particular a Banach lattice.
Then $L^\infty(\R, [L^2,\mathcal{H}_w]_\alpha)$ is a Banach lattice and hence a subspace of a space with local unconditional structure.
It therefore only remains to show \eqref{equ-local-unconditional-structure}.
To this end, consider 
\[j_0 : L^2 \to L^\infty(\R,L^2),\: f \mapsto (a \mapsto f \chi_{(-\infty,a]})\] and 
\[j_1 : X_w \to L^\infty(\R,\mathcal{H}_w),\: f \mapsto (a \mapsto f \chi_{(-\infty,a]}).\]
By definition of $X_w,$ $j_1$ is an isometric embedding, and it is easy to see that $j_0$ also is an isometric embedding.
Both $j_0$ and $j_1$ have a left inverse.
Indeed, for $g \in L^2,$ let $P_g$ be the linear form 
\[P_g : L^\infty(\R,L^2) \to \C,\: (a \mapsto f_a) \mapsto \lim_{a \to \infty} \langle f_a, g \rangle,\]
where $\lim_{a \to \infty}$ is any Banach limit (along integers $a \in \N,$ say).
It is easy to check that the linear form $P(f_a) : L^2 \to \C,\: g \mapsto P_g(f_a)$ is bounded by $\|f_a\|_{L^\infty(\R,L^2)} \| g \|_{L^2},$ so that by the Riesz representation theorem, we have constructed a contraction $P : L^\infty(\R,L^2) \to L^2,\: (f_a) \mapsto P(f_a).$
Further, it is easy to check that $P \circ j_0$ is the identity on $L^2.$
For $j_1,$ consider the analogous construction 
\[Q_g : L^\infty(\R,\mathcal{H}_w) \to \C,\: (a \mapsto f_a) \mapsto \lim_{a \to \infty} \langle f_a, g \rangle_{\mathcal{H}_w},\]
and set $Q(f_a): g \mapsto Q_g(f_a).$
One has $\langle Q j_1(f), g \rangle_{\mathcal{H}_w} = \lim_{a \to \infty} \langle f \chi_{(-\infty,a]}, g \rangle_{\mathcal{H}_w},$ and it is easy to check by Plancherel's theorem that for $f ,g \in L^2 \subseteq \mathcal{H}_w,$ this equals $\langle f , g \rangle_{\mathcal{H}_w}.$
Now conclude by density of $L^2$ in $\mathcal{H}_w$ that $Q j_1$ is the identity on $\mathcal{H}_w.$
Then by the retraction theorem \cite[Section 1.2.4]{Trieb}, we have that 
\[ X = [L^2,X_w]_\alpha \text{ is a complemented subspace of }[L^\infty(\R,L^2),L^\infty(\R,\mathcal{H}_w)]_\alpha.\]
Note that $j_0$ resp. $j_1$ take actually values in $L^\infty_0(\R,L^2)$ resp. $L^\infty_0(\R,\mathcal{H}_w),$ the closure of those functions that can be written as a step function $v(a) = \sum_{k = 1}^N \chi_{A_k}(a) f_k,$ with $A_k \subseteq \R$ measurable and $f_k \in L^2$ resp. $\in \mathcal{H}_w.$
Indeed, using dominated convergence, it is easy to check that 
\[v(a) = \chi_{[\frac1n (N+1) ,\infty)}(a) f + \sum_{k = -N}^N \chi_{[\frac1n k, \frac1n (k+1))}(a) f \chi_{(-\infty,\frac1n k]}\]
approximates $j_0(f)$ in $L^\infty(\R,L^2)$ as $n,N \to \infty.$
Thus, $j_0(L^2) \subseteq L^\infty_0(\R,L^2).$
Since $L^2 \hookrightarrow \mathcal{H}_w,$ one also has $j_1(L^2) \subseteq L^\infty_0(\R,\mathcal{H}_w),$ so by density of $L^2$ in $\mathcal{H}_w,$ $j_1(\mathcal{H}_w) \subseteq L^\infty_0(\R,\mathcal{H}_w).$
We infer that $X$ is a complemented subspace of $[L^\infty_0(\R,L^2),L^\infty_0(\R,\mathcal{H}_w)]_\alpha.$
We consider the $L^\infty_0$ spaces on $\R$ with finite measure, say, $\frac{1}{1 + a^2} da,$ noting that any change of measure results in an isometry on the $L^\infty$ level.
Then the simple functions (step functions such that the $A_k$ have finite measure) with values in $L^2$ are dense in $L^\infty_0(\R,L^2) \cap L^\infty_0(\R,\mathcal{H}_w) = L^\infty_0(\R,L^2).$
We infer that the proof of \cite[5.1.2 Theorem]{BeL} applies, showing that the natural map $[L^\infty_0(\R,L^2),L^\infty_0(\R,\mathcal{H}_w)]_\alpha \hookrightarrow L^\infty(\R,[L^2,\mathcal{H}_w]_\alpha)$ is an isomorphic embedding onto a closed subspace.
We have shown \eqref{equ-local-unconditional-structure} and the proof of the theorem is complete.
\end{proof}

\subsection{The bounded $\HI$ calculus assumption in Theorems \ref{Thm Sufficient conditions BIP} and \ref{Thm Sufficient conditions Sgr} is necessary}~\\

\noindent a)
The bounded $\HI$ calculus in the assumption of Theorem \ref{Thm Sufficient conditions BIP} cannot be omitted as the following example of a $0$-sectorial operator with uniformly bounded imaginary powers, but without any bounded $\HI$ calculus shows.
By Lemma \ref{Lem Elementary Mih Hor}, we have $\HI(\Sigma_\sigma) \hookrightarrow \Hor^\beta_r$ for any $\sigma \in (0,\pi),\:r \in [1,\infty)$ and $\beta > \frac1r$ so that this operator cannot have a $\Hor^\beta_r$ calculus.

Consider the $0$-sectorial operator $A$ on $L^p(\R)$ such that $A^{it}$ is the shift group, i.e. $A^{it}g = g(\cdot+t).$
By \cite[Lemma 5.3]{CDMY}, $A$ does not have an $\HI(\Sigma_\sigma)$ calculus to any positive $\sigma,$ unless $p = 2.$
However, $A^{it}$ is even uniformly bounded, so that the assumptions of Theorem \ref{Thm Sufficient conditions BIP} (2) hold for $\alpha = 0,$ and also for (1) according to Remark \ref{rem-semi-R-bounded}.\\

\noindent b)
The bounded $\HI$ calculus in the assumption of Theorem \ref{Thm Sufficient conditions Sgr} cannot be omitted.
We give an example below of a $0$-sectorial operator without an $\HI(\Sigma_\sigma)$ calculus for any $\sigma \in (0,\pi)$, but satisfying the assumptions of Theorem \ref{Thm Sufficient conditions Sgr} part (1) or (2) with $\alpha = 1,$ on a Hilbert space.
By Lemma \ref{Lem Elementary Mih Hor}, we have $\HI(\Sigma_\sigma) \hookrightarrow \Hor^\beta_r$ for any $\sigma \in (0,\pi),\:r \in [1,\infty)$ and $\beta > \frac1r,$ so that this operator cannot have a $\Hor^\beta_r$ calculus.

Our example will satisfy
\begin{equation}\label{Equ Sgr counterexample}
\|T(e^{i\theta} t)\| \lesssim ( \frac{\pi}{2} - |\theta| )^{-1} \quad (t > 0,\, |\theta| < \frac{\pi}{2}).
\end{equation}
In \cite[Theorem 4.1]{LM}, the following situation is considered, based on an idea of Baillon and Cl\'{e}ment.
Let $X$ be an infinite dimensional space admitting a Schauder basis $(e_n)_{n \geq 1}.$
Let $V$ denote the span of the $e_n$'s.
For a sequence $a = (a_n)_{n \geq 1},$ the operator $T_a : V \to V$ is defined by letting $T_a(\sum_n \alpha_n e_n) = \sum_n a_n \alpha_n e_n$
for any finite family $(\alpha_n)_{n \geq 1} \subset \C.$
Let $a^{(N)}=(a^{(N)}_n)_{n \geq 1}$ be the sequence defined by $a^{(N)}_n = \delta_{n \leq N}.$
It is well-known that for any Schauder basis (even conditional), $T_{a^{(N)}}$ extends to a bounded projection on $X$ and $\sup_N \|T_{a^{(N)}}\| < \infty$ \cite[Chapter 1]{LTz1}.
This readily implies that for any sequence $a = (a_n)_{n \geq 1}$ of bounded variation, $T_a$ extends to a bounded operator, and
\begin{equation}\label{Equ 1 Rem Norm cond bad Mihlin}
\|T_a\| \lesssim \|a\|_{BV} := \|a\|_{\ell^\infty} + \sum_{n = 1}^\infty |a_n - a_{n+1}|.
\end{equation}
In \cite{LM}, it is shown that for $a_n = 2^{-n},$ the bounded linear extension $A: X \to X$ of $T_a$ is a $0$-sectorial (injective) operator, and
that for $f \in \HI(\Sigma_\sigma),\,V$ is a subset of $D(f(A)).$
Further, for $x \in V,$ one has
\begin{equation}\label{Equ 2 Rem Norm cond bad Mihlin}
f(A)x = T_{f(a)}x,
\end{equation}
where $f(a)_n = f(a_n).$
Finally, it is shown in \cite{LM} that if the Schauder basis is conditional, then $A$ does not have a bounded $\HI$ calculus.

Now assume that $X$ is a separable Hilbert space, so that the assumptions of Theorem \ref{Thm Sufficient conditions Sgr} (1) and (2) reduce to norm boundedness in \eqref{Equ Sgr counterexample}.
Clearly, $X$ admits a Schauder basis, and as mentioned in \cite{LM} even a conditional one.
We take a conditional basis and consider the operator $A$ above without a bounded $\HI$ calculus.
By \eqref{Equ 1 Rem Norm cond bad Mihlin} and \eqref{Equ 2 Rem Norm cond bad Mihlin}, \eqref{Equ Sgr counterexample} will follow from
\begin{equation}\label{Equ 3 Rem Norm cond bad Mihlin}
\| (\exp( - t e^{i\theta} 2^{-n}))_n \|_{BV} \lesssim (\frac{\pi}{2}-|\theta|)^{-1} \quad (t > 0,\, |\theta| < \frac{\pi}{2}).
\end{equation}
It is easy to check that
\[|\exp(-t e^{i \theta} 2^{-n}) - \exp(-te^{i\theta} 2^{-(n+1)})| \lesssim 2^{-(n+1)}t \exp(-t \cos(\theta)2^{-(n+1)}).\]
Thus,
\begin{align*}
\| (\exp( - t e^{i\theta} 2^{-n}))_n \|_{BV} & \lesssim 1 + \sum_{n = 1}^\infty 2^{-(n+1)}t \exp(-t \cos(\theta)2^{-(n+1)}) \\
& \lesssim 1 + \int_0^1 st \exp(-t \cos(\theta)s) \frac{ds}{s} \\
& \lesssim 1 + (\cos\theta)^{-1} \int_0^\infty \exp(-s) ds \\
& \cong (\frac{\pi}{2} - |\theta|)^{-1}.
\end{align*}
This shows \eqref{Equ 3 Rem Norm cond bad Mihlin}, and thus $A$ satisfies \eqref{Equ Sgr counterexample} without having an $\HI$ calculus.

\subsection{The order of the resulting $\Hor^\beta_2$ calculus in Theorems \ref{Thm Sufficient conditions BIP}, \ref{Thm Sufficient conditions Sgr} and \ref{Thm Sufficient conditions Wave} cannot be improved in general}
\label{subsec-loss-alpha}

In Theorems \ref{Thm Sufficient conditions BIP}, \ref{Thm Sufficient conditions Sgr} and \ref{Thm Sufficient conditions Wave},
we assume one of several conditions on the differentiation parameter $\alpha$ and conclude the boundedness for a $\Hor^\beta_p$ calculus with some $\beta > \alpha.$
We give now a list of examples showing that the loss from $\alpha$ to $\beta$ is close to being optimal.\\

\noindent \textit{a) The gap between assumptions on the parameter $\alpha$ in (1) and (2) of Theorems \ref{Thm Sufficient conditions BIP}, \ref{Thm Sufficient conditions Sgr} and \ref{Thm Sufficient conditions Wave}, and the order $\alpha + \frac12$ of the $R$-bounded $\Hor^{\alpha+\frac12}_2$ calculus is optimal, even in a Hilbert space.}

\begin{proof}
Let $\alpha = m \in \N$ be given and consider the Jordan block
\[ B =
\left(\begin{array}{llll}
0 & 1 &  & 0 \\
 & 0 & \ddots &  \\
 &  & \ddots & 1 \\
0 &  &  & 0
       \end{array}
\right) \in \C^{(m+1)\times (m+1)}
\]
as an operator on the Hilbert space $X = \ell^2_{m+1}$ and set $A = e^B.$
Then
\[f \circ \log(A) = \left(\begin{array}{llll}
f(0) & \frac{f'(0)}{1!} & \ldots & \frac{f^{(m)}(0)}{m!} \\
0 & \ddots & \ddots &
                \end{array}\right),
\]
at least for $f \circ \log \in \bigcup_{\omega \in (0,\pi)}\Hol(\Sigma_\omega).$
Thus, for $x = (x_0,\ldots,x_m)$ and $y = (y_0,\ldots,y_m) \in \ell^2_{m+1},$
\[ |\spr{A^{it} x}{y}| = |\sum_{k=0}^m \sum_{l=k}^m \frac{(it)^{l-k}}{(l-k)!} x_l y_k| \lesssim \langle t \rangle^m \|x\|\,\|y\|,\]
and taking $x = (0,\ldots,0,1),\,y = (1,0,\ldots,0)$ shows that the exponent $m$ is optimal in this estimate.
Thus, $t \mapsto \langle t \rangle^{-\beta} \spr{A^{it} x}{y}$ belongs to $L^2(\R)$ for all $x,y\in X$ if and only if $\beta > m + \frac12,$
so that $A$ cannot have a $\Hor^\beta_2$ calculus for $\beta \leq m + \frac12.$
We refer to \cite{KrW2} for an adequate adaptation of the $R$-boundedness notion which is equivalent to the $R$-bounded $\Hor^\beta_2$ calculus.

On the other hand, $\|A^{it}\| \cong \langle t \rangle^m,$ so that the assumption in (2) of Theorem \ref{Thm Sufficient conditions BIP} holds with $\alpha = m$, and since $X$ is a Hilbert space, also the assumption in (1).
Furthermore the assumptions in (1) and (2) of Theorem \ref{Thm Sufficient conditions Wave} and thus, by Proposition \ref{Prop N_T N_W almost equivalent}, also of Theorem \ref{Thm Sufficient conditions Sgr} hold, because
\[\|(1+|t|A)^{-m} e^{itA}\| \cong | \frac{d^m}{ds^m}\left[ (1+|t|e^s)^{-m} \exp(ite^s)\right]_{s=0}| \leq C.\]
Therefore, in this example, assumptions (1) and (2) of Theorems \ref{Thm Sufficient conditions BIP}, \ref{Thm Sufficient conditions Sgr} and \ref{Thm Sufficient conditions Wave} hold with $\alpha = m,$ but $A$ does not have an $R$-bounded $\Hor^{m+\frac12}_2$ calculus.
This shows that part (1) of Theorems \ref{Thm Sufficient conditions BIP}, \ref{Thm Sufficient conditions Sgr} and \ref{Thm Sufficient conditions Wave} is sharp.
\end{proof}

\noindent\textit{b) The difference of $\alpha$ in the $R$-bounded $\Hor^\alpha_r$ calculus and assumption (2) in Theorem \ref{Thm Sufficient conditions BIP}.}

\begin{proof}
The $R$-bounded $\Hor^\alpha_r$ calculus implies essentially any of the assumptions in (1) and (2) of Theorems \ref{Thm Sufficient conditions BIP}, \ref{Thm Sufficient conditions Sgr} and \ref{Thm Sufficient conditions Wave} with the same parameter $\alpha,$ but in the converse direction, we have to increase the order of the H\"ormander calculus by essentially $\max(\frac12, \frac{1}{\type X} - \frac{1}{\cotype X}).$
An example showing the optimality of this gap in the case of Theorem \ref{Thm Sufficient conditions BIP} is again a multiplication operator, on a different space, it can be found in \cite[Theorem 5.26]{Kr}.
\end{proof}

\section{Bochner-Riesz means and the $\Hor^\alpha_1$ calculus}\label{Sec Bochner-Riesz}

For $\alpha > 1$ and $u > 0,$ we let $R^{\alpha-1}_u(\lambda) = (1 - \lambda/u)_+^{\alpha - 1}$ be the Bochner-Riesz functions, where $t_+ = \max(t,0).$
In the next theorem, we show that under some mild assumptions that in the presence of an $\HI$ calculus, $R$-boundedness of Bochner-Riesz means and an $R$-bounded $\Hor^\alpha_1$ calculus are essentially equivalent.

\begin{thm}\label{Thm Bochner Riesz}
Let $1 < \beta < \alpha$ and assume that $A$ is $0$-sectorial and has an auxiliary calculus $\Phi_A : \Hor^\beta_1 \to B(D(\theta),X)$ for some $\theta \geq 0,$ and an $\HI(\Sigma_\omega)$ functional calculus for some $\omega \in (0,\pi/2).$
\begin{enumerate}
\item If $A$ has an $R$-bounded $\Hor^\beta_1$ functional calculus, then
$ \{ R^{\alpha - 1}_u(A) :\: u > 0 \}$ is $R$-bounded in $X.$
\item If $ \{ R^{\alpha - 1}_u(A) :\: u > 0 \}$ is $R$-bounded in a Banach space $X$ (resp. in a space with property $(\alpha)$), then $A$ has a bounded (resp. $R$-bounded) $\Hor^\alpha_1$ functional calculus.
\end{enumerate}
\end{thm}

\begin{proof}
(1) It clearly suffices to note that $\sup_{u > 0} \|R_u^{\alpha - 1}\|_{\Hor^\beta_1}$ is finite, which has been proven in Lemma \ref{Lem Mihlin norms of functions} (4).\\

\noindent
(2) Note that the auxiliary calculus plus the $\HI(\Sigma_\omega)$ calculus allows to define $R_u^{\alpha-1}(A)$ as closed operators with dense domain.
 We want to apply Theorem \ref{Thm Localization Principle} and check condition \eqref{Equ Localized R-bound}.
So let $f \in C^\infty_c(\R_+)$ with $\supp f \subset [\frac12, 2]$ and $\|f\|_{\Sobolev^\alpha_1} \leq 1.$
Let $f^{(\alpha)}$ be the Cossar-Riemann-Liouville derivative of $f,$ as defined e.g. in
\cite[p.~1011]{GaMi}.
For functions with compact support, one has $\hat{f^{(\alpha)}}(\xi) = (-i\xi)^\alpha \hat{f}(\xi),\:\xi \in \R,$ so that in particular $\|f^{(\alpha)}\|_{L^1} \lesssim \|f\|_{\Sobolev^\alpha_1}.$
This also implies that $(f(a \cdot))^{(\alpha)}(t) = a^\alpha f^{(\alpha)}(at).$
By Lemma \ref{Lem Sobolev calculus} (4), for $x \in D_A,$
\[ f(2^n A) x = \frac{(-1)^m}{\Gamma(\alpha)} \int_0^\infty 2^{n\alpha} f^{(\alpha)}(2^n u) (u- A)_+^{\alpha - 1} x du = \frac{(-1)^m}{\Gamma(\alpha)} \int_0^\infty 2^{n\alpha} f^{(\alpha)}(2^{n}u) u^{\alpha - 1} R^{\alpha - 1}_u(A) x du , \]
where $m = \lfloor \alpha \rfloor.$
According to the assumptions, $\{R^{\alpha - 1}_u(A):\: u>0 \}$ is $R$-bounded, so that $\{ f(2^n A):\: n \in \Z,\: f\text{ as above} \}$ is $R$-bounded if $\|2^{n\alpha} f^{(\alpha)}(2^{n}u) u^{\alpha - 1}\|_{L^1(du)}$ can be bounded independently of $n \in \Z$ and $f$ as above.
But we have, since $\alpha > 1,$
\begin{align*} \int_0^\infty |2^{n\alpha} f^{(\alpha)}(2^{n} u) u^{\alpha - 1}| du
& = \int_0^\infty |f^{(\alpha)}(u) u^{\alpha -1}| du = \int_0^2 |f^{(\alpha)}(u)| u^{\alpha - 1} du \lesssim \int_0^2 |f^{(\alpha)}(u)| du \\
& \lesssim \|f\|_{\Sobolev^\alpha_1} \leq 1.
\end{align*}
\end{proof}

Note that the auxiliary calculus in the preceding Theorem was only assumed to be able to define the Bochner-Riesz means as closed densely-defined operators, it can be omitted if $A$ is self-adjoint on $L^2(U)$ and $X = L^p(U).$
To our knowledge, $R$-bounded Bochner-Riesz means have been considered for the first time in \cite{BoCl} and \cite{Stem}.
In \cite[Theorem A]{Stem}, see also \cite[Th\'eor\`eme (7.2)]{BoCl}, a functional calculus similar to our $\Hor^\beta_1$ calculus is established,
with the somewhat stronger norm,
\[\|f\|_{L^\infty(\R_+)} + \max_{j = 1,2,\ldots,\beta}\sup_{R > 0} \frac{1}{R} \int_0^R \lambda^j |f^{(j)}(\lambda)| d\lambda,\]
in place of $\max_{j=0,1,\ldots,\beta} \sup_{R > 0} \frac{1}{R} \int_{R/2}^R \lambda^j |f^{(j)}(\lambda)| d\lambda.$
It is assumed there that the Bochner-Riesz means of a self-adjoint operator $A$ are $R$-bounded on $L^p,\:1 < p < \infty$ and the semigroup generated by $-A$ is contractive on all $L^p,\:1 \leq p \leq \infty.$
Note that a self-adjoint operator generating such a semigroup always has an $\HI$ calculus on $L^p,\: 1 < p < \infty$ \cite[Theorem 2]{Co}.

In the following proposition, we give a sufficient criterion for the $R$-boundedness of Bochner-Riesz means.

\begin{prop}\label{Prop N_T to Bochner-Riesz}
Let $A$ be a $0$-sectorial operator and $\nu > \alpha \geq 0.$
Assume that 
\[ \left\{ \left( \frac{\Re z}{|z|} \right)^\alpha e^{-zA} :\: \Re z > 0 \right\}\text{ is }R\text{-bounded.}\]
Then $\{ R^\nu_u(A) :\: u > 0 \}$ is $R$-bounded, too.
\end{prop}

\begin{proof}
According to \cite[(4.2) and p.~332 Remark]{GaPy}, see also \cite[Proof of Lemma 6.1]{GaPy}, under the hypotheses of the proposition,
one has the formula
\[ R^\nu_u(A) = u^{-\nu} \frac{1}{2\pi i} \int_{\Re z = \frac1u} \frac{e^{-zA}}{z^{\nu+1}} e^{u z} dz. \]
If we write $z = \frac1u + is,$ then $|e^{uz}| = |e^{1+ isu}| = e,$ $|z^{\nu + 1}| = u^{-\nu - 1} |1+ius|^{\nu + 1}$ and $u^\alpha |z|^\alpha = |1 + ius|^\alpha.$
Hence 
\begin{equation}\label{Equ Proof Prop N_T to Bochner-Riesz}
R^\nu_u(A) = \frac{1}{2\pi i} \int_{\Re z= \frac1u} \left[ e^{-zA} \left(\frac{\Re z}{|z|}\right)^\alpha \right] g_u(z) dz
\end{equation}
with $g_u(z)= u^{-\nu} \frac{e^{uz}}{z^{\nu + 1}} \left( \frac{|z|}{\Re z} \right)^\alpha.$
We have $|g_u(z)| = e |1 + ius|^{-\nu - 1} |1 + ius|^\alpha u.$
Since $\nu > \alpha,$ $\int_{\Re z = \frac1u} |g_u(z)| d |z| \leq e \int_\R |1 + ius|^{-(\nu-\alpha+1)} u ds \leq C$
with $C$ independent of $u.$
Now the proposition follows from \eqref{Equ Proof Prop N_T to Bochner-Riesz} together with the assumptions and \cite[Corollary 2.14]{KuWe}.
\end{proof}

\section{Bisectorial and Strip-type Operators}\label{Sec Strip-type Operators}

\subsection{Bisectorial operators}

In this short subsection we indicate how to extend our results to bisectorial operators.
An operator $A$ with dense domain on a Banach space $X$ is called bisectorial of angle $\omega \in [0,\frac{\pi}{2})$ if it is closed, its spectrum is contained in the closure of $S_\omega = \{ z \in \C :\: |\arg(\pm z)| < \omega \},$ and one has the resolvent estimate 
\[\|(I+\lambda A)^{-1}\|_{B(X)} \leq C_{\omega'},\: \forall \: \lambda \not\in S_{\omega'},\: \omega' > \omega .\]
If $X$ is reflexive, then for such an operator we have again a decomposition $X = N(A) \oplus \overline{R(A)},$ so that we may assume that $A$ is injective.
The $\HI(S_\omega)$ calculus is defined as in \eqref{Equ Cauchy Integral Formula}, but now we integrate over the boundary of the double sector $S_\omega.$
If $A$ has a bounded $\HI(S_\omega)$ calculus, or more generally, if we have $\|Ax\| \cong \|(-A^2)^{\frac12}x\|$ for $x \in D(A) = D((-A^2)^{\frac12})$ (see e.g. \cite{DW}), then the spectral projections $P_1,\: P_2$ with respect to $\Sigma_1 = S_\omega \cap \C_+,\: \Sigma_2 = S_\omega \cap \C_-$ give a decomposition $X = X_1 \oplus X_2$ of $X$ into invariant subspaces for the resolvents of $A$ such that the restriction $A_1$ of $A$ to $X_1$ and $-A_2$ of $-A$ to $X_2$ are sectorial operators with $\sigma(A_i) \subset \Sigma_i.$
For $f \in \HI_0(S_\omega)$ we have 
\begin{equation}\label{Equ Bisectorial}
f(A)x = f|_{\Sigma_1}(A_1)P_1 x + f|_{\Sigma_2}(A_2)P_2 x.
\end{equation}
We define the H\"ormander class $\Ha(\R)$ on $\R$ by restrictions: $f \in \Ha(\R)$ iff $f \chi_{\R_+} \in \Ha$ and $f(-\cdot) \chi_{\R_+} \in \Ha.$
Let $A$ be a $0$-bisectorial operator, i.e. $A$ is $\omega$-bisectorial for all $\omega > 0.$
Then $A$ has a $\Ha(\R)$ calculus if there is a constant $C$ so that $\|f(A)\| \leq C \|f\|_{\Ha(\R)}$ for $f \in \bigcap_{0 < \omega < \pi} \HI(S_\omega) \cap \Ha(\R).$
Clearly, $A$ has a $\Ha(\R)$ calculus if and only if $A_1$ and $-A_2$ have a $\Ha$ calculus and in this case \eqref{Equ Bisectorial} holds again.

Let $f_t(\lambda) = \begin{cases} \lambda^{it} :\: \Re \lambda > 0 \\ (-\lambda)^{it} :\: \Re \lambda  < 0 \end{cases}.$
Then $f_t \in \HI(S_\omega)$ for any $\omega \in (0,\frac{\pi}{2}).$
Clearly, one has $f_t(A) = A_1^{it} \oplus (-A_2)^{it}$ on $X = X_1 \oplus X_2.$
Similarly, let $R_u^\alpha(\lambda) = ( 1 - |\lambda| / u )_+^\alpha,$ so that $R_u^\alpha(A) = ( 1 - A_1/u )_+^\alpha \oplus ( 1 - (-A_2)/u )_+^\alpha.$
Finally, let $g_s^\alpha(\lambda) = ( i + |s| \lambda )^{-\alpha} e^{i s \lambda},$ so that $g_s^\alpha(A) = ( i + |s| A_1 )^{-\alpha} e^{i s A_1} \oplus ( i - |s| (-A_2) )^{-\alpha} e^{-i|s|(-A_2)}.$
Then by an obvious modification of the proof of Proposition \ref{Prop N_T N_W almost equivalent}, one can show that if $\{ g_s^\alpha(A) :\: s \in \R \}$ is $R$-bounded then \eqref{R_T} holds for both $A_1$ and $-A_2,$ and that $\sup_{s \in \R}R(\{ g_{2^k s}^\alpha(A) :\: k \in \Z \}) < \infty$ implies that \eqref{N_T} holds for both $A_1$ and $-A_2.$

Then using the projections $P_1$ and $P_2,$ it is clear how our results from Sections \ref{Sec 6 Hormander}, \ref{sec-Semigroup} and \ref{Sec Bochner-Riesz} extend to bisectorial operators.

\subsection{Operators of strip-type}
For $\omega > 0$ we let $\Str_\omega = \{z \in \C :\: |\Im z| < \omega\}$ the horizontal strip of height $2 \omega.$
We further define $\HI(\Str_\omega)$ to be the space of bounded holomorphic functions on $\Str_\omega,$ which is a Banach algebra equipped with the norm $\|f\|_{\infty,\omega} = \sup_{\lambda \in \Str_\omega} |f(\lambda)|.$
A densely defined operator $B$ is called an operator of $\omega$-strip-type if $\sigma(B) \subset \overline{\Str_\omega}$ and for all $\theta > \omega$ there is a $C_\theta > 0$ such that $\|\lambda (\lambda - B)^{-1}\| \leq C_\theta$ for all $\lambda \in \overline{\Str_\theta}^c.$
Similarly to the sectorial case, one defines $f(B)$ for $f \in \HI(\Str_\theta)$ satisfying a decay for $|\Re \lambda| \to \infty$ by a Cauchy integral formula, and says that $B$ has a bounded $\HI(\Str_\theta)$ calculus provided that $\|f(B)\| \leq C \|f\|_{\infty,\theta}.$ 
In this case $f \mapsto f(B)$ extends to a bounded homomorphism $\HI(\Str_\theta) \to B(X).$
We refer to \cite{CDMY} and \cite[Chapter 4]{Haasa} for details.
We call $B$ of $0$-strip-type if $B$ is $\omega$-strip-type for all $\omega > 0.$

There is an analogous statement to Lemma \ref{Lem Hol} which holds for a $0$-strip-type operator $B$ and $\Str_\omega$ in place of $A$ and $\Sigma_\omega,$ and $\Hol(\Str_\omega) = \{ f : \Str_\omega \to \C :\: \exists n \in \N :\: (\rho \circ \exp)^n f \in \HI(\Str_\omega) \},$
where $\rho(\lambda) = \lambda ( 1 + \lambda)^{-2}.$

In fact, $0$-strip-type operators and $0$-sectorial operators with bounded $\HI(\Str_\omega)$ and bounded $\HI(\Sigma_\omega)$ calculus are in one-one correspondence by the following lemma. 
For a proof we refer to \cite[Proposition 5.3.3., Theorem 4.3.1 and Theorem 4.2.4, Lemma 3.5.1]{Haasa}.

\begin{lem}
\label{lem-strip-type}
Let $B$ be an operator of $0$-strip-type and assume that there exists a $0$-sectorial operator $A$ such that $B = \log(A)$.
This is the case if $B$ has a bounded $\HI(\Str_\omega)$ calculus for some $\omega < \pi.$
Then for any $f \in \bigcup_{0 < \omega < \pi} \Hol(\Str_\omega)$ one has
\[ f(B) = (f\circ \log)(A). \]
Note that the logarithm belongs to $\Hol(\Sigma_\omega)$ for any $\omega \in (0,\pi).$
Conversely, if $A$ is a $0$-sectorial operator that has a bounded $\HI(\Sigma_\omega)$ calculus for some $\omega \in (0,\pi),$ then $B = \log(A)$ is an operator of $0$-strip-type.
\end{lem}

Let $B$ be an operator of $0$-strip-type, $p \in [1,\infty)$ and $\alpha > \frac{1}{p}.$
We say that $B$ has a (bounded) $\Sobolev^\alpha_p$ calculus if there exists a constant $C > 0$ such that
\[ \|f(B)\| \leq C \|f\|_{\Sobolev^\alpha_p} \quad (f\in \bigcap_{\omega > 0} \HI(\Str_\omega) \cap \Sobolev^\alpha_p) .\]
In this case, by density of $\bigcap_{\omega > 0} \HI(\Str_\omega) \cap \Sobolev^\alpha_p$ in $\Sobolev^\alpha_p$,
the definition of $f(B)$ can be continuously extended to $f \in \Sobolev^\alpha_p.$

Assume that $B$ has a $\Sobolev^\alpha_p$ calculus for some $\alpha > \frac{1}{p}.$
Let $f : \R \to \C$ be a function such that $t \mapsto f(t) e^{\epsilon t} (1 + e^{\epsilon t})^{-2}$ belongs to $\Sobolev^\alpha_{p}$ for some $\epsilon > 0.$
We define the operator $f(B)$ to be the closure of
\[ \begin{cases}
D_B \subset X & \longrightarrow X \\
x & \longmapsto \sum_{n \in \Z} (\equi_n f)(B)x,
\end{cases}
\]
where $D_B = \{ x \in X :\: \exists N \in \N:\: \equi_n(B)x= 0 \quad (|n| \geq N) \}$ and $(\equi_n)_{n \in \Z}$ is an equidistant partition of unity (see Lemma \ref{Lem Classical and modern Hoermander condition}).

Then there is a version of Lemma \ref{Lem Soaloc calculus}, for which a proof can be found in \cite[Proposition 4.25]{Kr}.
Let $\Wa = \{ f \in L^p_{\text{loc}}(\R) : \: \|f\|_{\Wa} = \sup_{n \in \Z} \| \equi_n f \|_{\Soa} < \infty \}.$
Note that the H\"ormander class $\Wa$ is covered by the variant of Lemma \ref{Lem Soaloc calculus}.
Let $p \in (1,\infty),\: \alpha > \frac1p$ and let $B$ be an operator of $0$-strip-type.
We say that $B$ has a (bounded) $\Wa$ calculus if there exists a constant $C > 0$ such that
\[ \|f(B)\| \leq C \|f\|_{\Wa} \quad (f \in \bigcap_{\omega > 0} \HI(\Str_\omega) \cap \Wa ). \]
The strip-type version of the main Theorem \ref{Thm Sufficient conditions BIP} reads as follows.

\begin{thm}\label{Thm Vertical chain of conditions strip}
Let $B$ be $0$-strip-type operator with $\HI$ calculus on some Banach space with property $(\alpha).$
Denote $U(t)$ the $C_0$-group generated by $iB.$
For $r \in (1,2]$ and $\alpha > \frac1r,$ consider the condition
\begin{equation}
B \text{ has an }R\text{-bounded }\Wor^\alpha_r\text{ calculus.}
\tag*{$(C_r)_\alpha$}
\end{equation}
Furthermore, for $\alpha \geq 0,$ we consider the conditions
\begin{enumerate}
\item[(I${)_\alpha}$] There exists $C > 0$ such that for all $t \in\R,\,\|U(t)\| \leq C (1+|t|)^\alpha.$
\vspace{0.2cm}
\item[(II${)_\alpha}$] The set $\{ \tma U(t):\:t \in \R\}$ is (semi-)$R$-bounded.
\vspace{0.2cm}
\end{enumerate}
Then the following hold.
\begin{enumerate}
\item[(a)] Let $r \in (1,2]$ such that $\frac1r > \frac{1}{\type X} - \frac{1}{\cotype X}$ and $\beta > \alpha + \frac1r.$
Then (I$)_\alpha$ implies $(C_r)_\beta.$
\item[(b)] Consider $\alpha,\beta \geq 0$ with $\beta > \alpha + \frac12.$
Then (II$)_\alpha$ implies $(C_2)_\beta.$
\end{enumerate}
\end{thm}

\begin{proof}
Considering the $0$-sectorial operator $A = e^B,$ the Theorem follows at once from the sectorial Theorem \ref{Thm Sufficient conditions BIP} together with Lemma \ref{lem-strip-type}.
\end{proof}

\end{document}